\documentclass[11pt,leqno]{article}
\usepackage{amsmath, amscd, amsthm, amssymb, graphics, mathrsfs, setspace, fancyhdr, times, bm, bbm, enumitem}
\usepackage[letterpaper,top=1.05in, bottom=1.05in, left=1.05in, right=1.05in]{geometry}
\usepackage[colorlinks=true,pagebackref=true]{hyperref} 
\hypersetup{backref,hidelinks}

\usepackage{etoolbox}
\usepackage{comment}
\usepackage[all,cmtip]{xy}

\usepackage{tikz}
\usetikzlibrary{matrix,arrows}
\usepackage{tikz-cd}
\usetikzlibrary{cd}

\usepackage{stmaryrd}

\usepackage[utf8]{inputenc}
\usepackage[T1]{fontenc}

\newtoggle{final}
\toggletrue{final}

\usepackage[textsize=scriptsize]{todonotes}
\iftoggle{final} {
\newcommand{\NB}[1]{}
\newcommand{\TODO}[1]{}
\newcommand{\tom}[1]{}
\newcommand{\mike}[1]{}
\renewcommand{\todo}[1]{}
}{ 
\newcommand{\NB}[1]{\todo[color=gray!40]{#1}}
\newcommand{\TODO}[1]{\todo[color=red]{#1}}
\newcommand{\tom}[1]{\todo[color=green]{#1}}
\newcommand{\mike}[1]{\todo[color=blue!20]{#1}}
}

\newcommand{\Et}{\mathrm{Et}}
\def\naive{\mathrm{naive}}
\newcommand{\comp}{{{\kern -.5pt}\wedge}}

\newcommand{\SL}[1]{\mathrm{SL}_{#1}}
\newcommand{\GL}{\mathrm{GL}}
\newcommand{\veff}{\mathrm{veff}}

\numberwithin{equation}{section}

\def\A{\mathbb A}
\def\P{\mathbb P}

\def\Alg{\mathrm{Alg}}

\def\CAlg{\mathrm{CAlg}}

\def\Id{\mathrm{Id}}

\newcommand{\R}{{\mathbb R}}
\newcommand{\C}{{\mathbb C}}
\newcommand{\Z}{{\mathbb Z}}
\newcommand{\Q}{{\mathbb Q}}
\newcommand{\F}{{\mathbb F}}
\newcommand{\id}{\operatorname{id}}

\newcommand{\cof}{\mathrm{cof}}

\newcommand{\aone}{{\mathbb A}^1}

\newcommand{\SH}{\mathrm{SH}}

\newcommand{\Spc}{\mathrm{Spc}}
\newcommand{\Sm}{\mathrm{Sm}}
\newcommand{\Ft}{\mathrm{Ft}}
\newcommand{\Nis}{\mathrm{Nis}}
\newcommand{\Shv}{\mathrm{Shv}}
\newcommand{\Fun}{\mathrm{Fun}}
\newcommand{\Gmp}[1]{{\mathbb{G}_m^{\wedge #1}}}
\newcommand{\Gm}{{\mathbb{G}_m}}
\newcommand{\ret}{{r\acute{e}t}}
\newcommand{\et}{{\acute{e}t}}
\renewcommand{\1}{\mathbbm{1}}

\newcommand{\Sch}{\mathrm{Sch}}
\newcommand{\op}[1]{\operatorname{#1}}
\newcommand{\Spec}{\op{Spec}}

\newcommand{\wequi}{\simeq}

\DeclareRobustCommand{\ul}{\underline}
\def\Map{\mathrm{Map}}
\def\iMap{\ul{\mathrm{Map}}}

\def\ph{\mathord-}

\DeclareMathOperator*{\colim}{colim}
\DeclareMathOperator*{\diag}{diag}

\let\lim=\relax
\DeclareMathOperator*{\lim}{lim}

\newcommand{\Addresses}{{
		\bigskip
		\footnotesize
		
		A.~Asok, Department of Mathematics, University of Southern California, 3620 S. Vermont Ave.,
		Los Angeles, CA 90089-2532, United States; \textit{E-mail address:} \url{asok@usc.edu}
		
		\medskip
		
		T.~Bachmann, Department of Mathematics, Johannes-Gutenberg-Universität, Staudingerweg 9, 55128 Mainz, Germany; \textit{E-mail address:} \url{tom.bachmann@zoho.com}
		\medskip
		
		E.~Elmanto, Department of Mathematics, Harvard University, One Oxford Street, Cambridge, MA 02138, United States \textit{E-mail address:} \url{elmanto@math.harvard.edu}
		\medskip
		
		M.J.~Hopkins, Department of Mathematics, Harvard University, One Oxford Street, Cambridge, MA 02138, United States \textit{E-mail address:} \url{mjh@math.harvard.edu}

}}

\newcounter{intro}
\theoremstyle{plain}
\newtheorem{theorem}{Theorem}[section]

\newtheorem{lem}[theorem]{Lemma}
\newtheorem{cor}[theorem]{Corollary}
\newtheorem{proposition}[theorem]{Proposition}
\newtheorem*{claim*}{Claim} 

\newtheorem*{thm*}{Theorem}
\newtheorem*{problem*}{Problem}

\theoremstyle{definition}
\newtheorem{defn}[theorem]{Definition}

\theoremstyle{remark}
\newtheorem{rem}[theorem]{Remark}

\newtheorem{ex}[theorem]{Example}

\numberwithin{equation}{subsection}

\begin{document}
\pagestyle{fancy}
\renewcommand{\sectionmark}[1]{\markright{\thesection\ #1}}
\fancyhead{}
\fancyhead[LO,R]{\bfseries\footnotesize\thepage}
\fancyhead[LE]{\bfseries\footnotesize\rightmark}
\fancyhead[RO]{\bfseries\footnotesize\rightmark}
\chead[]{}
\cfoot[]{}
\setlength{\headheight}{1cm}	
	
\title{{\bf Unstable motivic and real-{\'e}tale homotopy theory}}
\date{\today}

\author{Aravind Asok\thanks{Aravind Asok was partially supported by National Science Foundation Awards DMS-1802060 and DMS-2101898} \and Tom Bachmann \and Elden Elmanto \and Michael J. Hopkins}
\maketitle

\begin{abstract}
We prove that for any base scheme $S$, real étale motivic (unstable) homotopy theory over $S$ coincides with unstable semialgebraic topology over $S$ (that is, sheaves of spaces on the real spectrum of $S$).
Moreover we show that for pointed connected motivic spaces over $S$, the real étale motivic localization is given by smashing with the telescope of the map $\rho: S^0 \to \Gm$.
\end{abstract}

\begin{footnotesize}
\setcounter{tocdepth}{1}
\tableofcontents
\end{footnotesize}

\section{Introduction}
The real étale topology is a Grothendieck topology on the category of schemes that is closely tied to {\em semi-algebraic geometry}, i.e., polynomial equations and inequalities or, alternatively, rings and fields equipped with orderings.  After its introduction by Coste and Roy \cite{MR653174}, this topology has been studied extensively, notably in Scheiderer's seminal work \cite{real-and-etale-cohomology}.

Motivic homotopy theory, invented by Morel and Voevodsky, is a homotopy theory for (usually smooth) schemes over a base equipped with a Grothendieck topology.  Typically, this Grothendieck topology is taken to be the Nisnevich topology, but numerous results have pointed the way to consideration of what one mighgt call ``designer'' Grothendieck topologies, e.g., the cdh topology on all schemes, which is particularly useful in the presence of resolution of singularities.

This paper is concerned with studying motivic homotopy theory using the real-\'etale topology on smooth schemes.  In that context, the groundwork for this paper was laid by the second author in \cite{bachmann-real-etale}, where {\em stable} motivic homotopy theory was considered using the real-\'etale topology.  That paper observed that the real-\'etale topology interacted extremely well with the ``real realization'' functor \cite[\S3.3]{A1-homotopy-theory}, and, in fact, served to provide a computationally useful description of this functor.

To set the stage, recall that in studying ordered fields, a distinguished role is played by $-1$; in our context, this is encoded in the map $\rho: S^0 \to \Gm$ which sends the non-basepoint to $-1$.  In any ordered field, units can be classified as positive or negative leading to the idea that the map $\rho$ should be a homotopy equivalence in ``real-étale'' motivic homotopy theory.  Loosely speaking, the main result of \cite{bachmann-real-etale} asserts that inverting real-étale equivalences in stable motivic homotopy theory is essentially the same thing as inverting $\rho$; either way, the resulting homotopy theory is what one might call stable semialgebraic topology over the given base.

Here, we aim to ``destabilize'' the results of \cite{bachmann-real-etale}.  Write $\Spc(S)$ for the unstable motivic $\infty$-category (underlying the model category) of Morel--Voevodsky \cite[\S3.3]{A1-homotopy-theory}, realized as a sub-category of the $\infty$-category $\mathrm{P}(\Sm_S)$ of presheaves of spaces on $\Sm_S$.  We can then summarize our main results as follows.
\begin{theorem}[See Theorems \ref{thm:main} and \ref{thm:main-comparison-connected}.] \label{thm:intro}
Let $S$ be any scheme.
\begin{enumerate}[noitemsep,topsep=1pt]
\item If $X \in \Spc(S)$ is a sheaf in the real étale topology, then:
  \begin{enumerate}[noitemsep,topsep=1pt]
    \item $X$ is real étale locally constant, that is, is obtained by left Kan extension (and real-étale sheafification) from a sheaf on the small real-étale site of $S$.
    \item The map $\rho^*: \Omega_\Gm X \to X$ is an equivalence.
  \end{enumerate}
\item Locally constant real étale sheaves are $\A^1$-invariant, and consequently the subcategory of $\Spc(S)$ consisting of motivic spaces satisfying real-étale descent is equivalent to the category of locally constant real-étale sheaves on $S$.

\item If $X \in \Spc(S)_*$ is connected, then the telescope \[ X[\rho^{-1}] := \colim \left(X \xrightarrow{\rho} \Gm \wedge X \xrightarrow{\rho} \Gmp{2} \wedge X \xrightarrow{\rho} \dots\right) \] is a real-étale sheaf, and in fact is the initial ($\A^1$-invariant) real-étale sheaf under $X$.
\end{enumerate}
\end{theorem}

\begin{rem}[Analogs and hypotheses]
Sending a smooth scheme over $\R$ to its associated complex space equipped with action of $\mathrm{Gal}(\C/\R)$ extends to a realization functor from the category of motivic spaces over $\real$ to $C_2$-equivariant homotopy theory.  This realization functor can be used to elucidate the above theorem.
To begin with, let $a: S^0 \to S^\sigma \in \Spc^{C_2}_*$ be the inclusion of the fixed points of the sign representation sphere, that is, $C_2$-equivariant analog of $\rho$.
Then the telescope $S^0[a^{-1}]$ has fixed points $S^0$ and contractible underlying space.
It follows that for any pointed $C_2$-space $X$, the analog of real-étale localization is $X[a^{-1}]$.
This establishes the analog of part (3) of the above theorem in the $C_2$-equivariant world, even without the connectivity assumption.

In contrast, let $k=\C$ and consider $S^0[\rho^{-1}]$.
Since the presheaf $\Gmp{n}$ is already an $\A^1$-invariant Nisnevich sheaf (see \cite[Lemma 3.36]{A1-alg-top} for the sheaf property), we see that $S^0[\rho^{-1}]$ has non-contractible $\C$-points, and so this space is not real-étale local.  In other words, the connectivity assumption is necessary in order for the $\rho$-periodization to coincide with the real-étale localization.
\end{rem}


\subsection*{Proof structure}
The proof of Theorem~\ref{thm:intro} consists of several arguments of rather different flavor, which are largely independent of each other, so we take a moment to discuss some of the required ingredients here.
\begin{enumerate}[noitemsep,topsep=1pt]
\item We show that $\rho$ becomes an equivalence in unstable real-étale motivic homotopy theory by providing an explicit homotopy.  The key insight is that the real-étale topology is not subcanonical: loosely speaking, this means that sheafification introduces many new functions.  For example, $1/(x^2+1)$ is a well-defined function on the real-étale version of $\A^1$.

\item We show that locally constant real-étale sheaves (of spaces) are $\A^1$-invariant by essentially arguing from first principles.  The argument reduces to considering the affine line over a real closed field, where the result is essentially a claim about the structure of endpoints of intervals.

\item To prove that all $\A^1$-invariant real-étale sheaves are locally constant, we effectively show that all smooth varieties over a real closed field are, locally for the real-étale topology, $\A^1$-contractible.  To illustrate how this happens, consider for example $X \in \Sm_S$ which is étale over $\A^1_S$.  Since étale morphisms are local homeomorphisms in the real-étale topology, $X$ can be covered by real-étale open subsets which are isomorphic to real-étale open subsets of $\A^1_S$.  For each point in such a subset $U$, locally around its image in $S$ we can find a section $s$ of the projection $p$ to $S$.  In that case, the subset $V$ of $U$ consisting of all points $x$ such that the interval in $\A^1_{px}$ joining $x$ and $spx$ lies in $U$.  Then, we argue that $V$ is a real-étale open subset, which by construction is $\A^1$-homotopy equivalent to a real-étale open subset of $S$.  This observation can then be fed into an argument proceedingly inductively on the dimension of the variety.

\item To show that the $\rho$-telescope on a connected motivic space satisfies real-étale descent, we employ a variety of techniques.  One key idea is that if $l/k$ is a field extension (in characteristic $0$, say), then we can build a transfer map $\Sigma \Spec(k)_+ [\rho^{-1}] \to \Sigma \Spec(l)_+[\rho^{-1}]$ coming from the Morel transfer (a.k.a. Pontyragin--Thom construction) 
\[ 
S^1 \wedge \Gm \wedge \Spec(k)_+ \wequi \P^1 \to \P^1/\P^1 \setminus \Spec(l) \wequi S^1 \wedge \Gm \wedge \Spec(l)_+. 
\]
Careful analysis establishes that if $l/k$ is a real-étale cover, then an appropriate modification of this transfer is actually a section of the canonical projection.  Granted this, we may show that the suspended $\rho$-telescope of the Čech nerve of $\Spec(l) \to \Spec(k)$ is contractible and eventually reduce the descent assertion to this one.
\end{enumerate}

\subsection*{Organization}
We begin with several preliminary sections.
In \S\ref{sec:prelim1} we collect some preliminaries about motivic spaces: in \S\ref{subsec:contract} we show that contractibility of motivic spaces over a base can (often) be checked after pullback to the residue fields, and in the remaining subsections we provide a certain explicit model for the equivalence $\Gm \wedge \P^1 \wequi \SL{2}$.
In the next preliminary section, \S\ref{sec:james}, we recall some facts about relative \emph{James constructions} for which we could not find a suitable reference.  Given a map $\alpha: \1 \to A$ in a symmetric monoidal $\infty$-category, the James construction $J_\alpha$ is the free $\mathscr E_1$-algebra with unit factoring through $\alpha$.
We relate this, under favorable hypotheses, to both the symmetric monoidal localization at the map $\alpha$, and $\alpha$-telescopes.
In the final preliminary section \S\ref{sec:small-ret} we recall basic facts about the real-étale topology.  We then establish our main results in two more sections: in \S\ref{sec:main1} we prove parts (1) and (2) of Theorem \ref{thm:intro}, and in \S\ref{sec:motivic-real-etale-locn} we establish part (3).  We conclude with some complementary results and applications in \S\ref{sec:realizations}.

\subsection*{Acknowledgements}
We would like to thank Marc Levine for helpful comments.\todo{more?}

\section{Motivic spheres, contractibility and loop spaces} \label{sec:prelim1}
In this preliminary section, we collect some facts about motivic spaces that we could not find references for.

\subsection{Contractibility of motivic spaces} \label{subsec:contract}
It is well-known that as a consequence of the localization theorem, equivalences of motivic spectra can be checked pointwise (see e.g. \cite[Proposition B.3]{bachmann-norms}).
The following is an unstable analog, with essentially the same proof.

\begin{proposition} \label{prop:unstable-loc}
	Let $S$ be a scheme locally of finite Krull dimension.
	For $s \in S$ denote by $i_s: s \to S$ the inclusion.
	\begin{enumerate}[noitemsep,topsep=1pt]
		\item Let $X \in Spc(S)$.
		Then $X = *$ if and only if $i_s^*(X) = * \in Spc(s)$ for all $s \in S$.
		\item Let $f: X \to Y \in Spc(S)_*$.
		If $i_s^*(f)$ is an equivalence for all $s$, then $\Sigma f$ is an equivalence.
	\end{enumerate}
\end{proposition}
\begin{proof}
	(1) The condition is clearly necessary; we prove sufficiency.
	By Zariski descent we may assume $S$ of finite Krull dimension and qcqs (e.g. affine).
	We prove the result by induction on the dimension of $S$.
	By hypercompleteness \cite[Proposition A.3]{bachmann-norms} it suffices to prove that for $s \in S$, the restriction of $X$ to $S_s$ is contractible\footnote{Namely, viewing $X$ as a hypercomplete sheaf on $\Sm_S$, it suffices to show that for $Y \in \Sm_S$ and $y \in Y$, the stalk of $X$ at $Y_y$ is contractible. But if $s$ is the image of $y$ in $S$, then $Y_y \to S$ factors through $S_s$, whence $X(Y_y)$ is determined by $X|_{S_s}$.}; hence we may assume $S = S_s$.
	Write $j: U \to S$ for the inclusion of the complement of $s$; note that $U$ has smaller dimension than $S$.
	The localization theorem \cite[\S3 Theorem 2.21]{A1-homotopy-theory} supplies us with a pushout square
	\begin{equation*}
		\begin{CD}
			j_\sharp j^* X @>>> X \\
			@VVV          @VVV    \\
			U @>>> i_{s*} i_s^* X.
		\end{CD}
	\end{equation*}
	By induction $j^* X = *$ and hence the left hand vertical map is an equivalence.
	Consequently so is the right hand vertical map.
	Thus $X = *$, since $i^*_s X = *$ by assumption. \newline
	
	\noindent (2) Combining the assumption that $i_s^*(f)$ is an equivalence for all $s$ with the conclusion of (1), we deduce that $\cof(f) = *$
	Processing in the cofiber sequence attached to $f$ yields the cofiber sequence:
	\[ 
	\cof(f) \longrightarrow \Sigma X \stackrel{\longrightarrow}{\Sigma f} \Sigma Y, 
	\] 
	and since $\cof(f) = \ast$, we conclude $\Sigma f$ is an equivalence.
\end{proof}

\subsection{An equivalence of motivic spheres}
If $R$ is a commutative ring, then write $\mathrm{Proj}_r(R)$ for the set of isomorphism classes of rank $r$ projective $R$-modules.  Fix a base commutative ring $k$.  If $n$ is an integer, recall that ${\mathbb P}^n$ represents the functor that attaches to a commutative $k$-algebra $R$ the set of pairs $(P,\varphi)$ where $P \in \mathrm{Proj}_1(R)$ and $\varphi: R^{\oplus n+1} \to P$ is an epimorphism.  One defines $\op{J}_n$ as the functor on commutative $k$-algebras given by the formula:
\[
\op{J}_n(R) := \{ (P,\varphi,\psi) | (P,\varphi) \in {\mathbb P}^n(R), \psi: P \to R^{\oplus n+1}, \varphi \circ \psi = id_P\}.
\]
The functor $\op{J}_n$ admits a natural transformation to ${\mathbb P}^n$ by forgetting $\psi$ in a triple $(P;\varphi,\psi)$.  The following well-known result summarizes the basic properties of $\op{J}_n$.
\todo{proof/reference?}
\begin{lem}
	\label{lem:JdeviceofPnproperties}
	The following statements hold.
	\begin{enumerate}[noitemsep,topsep=1pt]
		\item The functor $\op{J}_n$ is the functor of points of the complement of the incidence hyperplane in ${\mathbb P}^n \times {\mathbb P}^n$.
		\item The morphism $\op{J}_n \to \mathbb{P}^n$ is an affine vector bundle torsor realizing the source as an affine vector bundle torsor over the target; in particular this morphism is an equivalence in $\Spc(k)$.
	\end{enumerate}
\end{lem}

Given an element $(P,\varphi,\psi) \in J_n(R)$, the splitting $\psi$ is equivalent to specifying an isomorphism $R^{\oplus n+1} \cong P \oplus \ker(\varphi)$.  The composite $\psi \circ \varphi \in \op{End}_R(R^{\oplus n+1})$ is a rank $1$ projection operator $\theta$ whose image is exactly $P$; we thus refer to $J_n$ as the scheme of rank $1$ projection operators.  

Define a morphism
\[
\Gm{}
\times \op{J}_n \longrightarrow \mathrm{GL}_{n+1}
\]
functorially by sending $(\lambda,(P,\varphi,\psi))$ to the invertible endomorphism of $P \oplus \ker(\varphi)$ given by $diag(\lambda,id_{\ker(\varphi)})$.  To produce a map with target $\SL{n+1}$ we can multiply by a further diagonal matrix (with respect to the standard basis). This yields a morphism $\Gm \times J_n \to \op{SL}_{n+1}$ which in terms of the associated projection operator $\theta$ can be written as
\[
R: (\lambda, \theta) \longrightarrow (\lambda \theta + (\Id - \theta))\cdot \op{diag}(\lambda^{-1},1).
\]
If we point $\op{J}_n$ with the projection operator $\op{diag}(1,0)$ corresponding to projection onto the first summand, and $\op{SL}_{n+1}$ with the identity matrix, then the morphism above is base-point preserving.  Moreover, direct computation using the formula above implies that $R(1,\theta) = \Id_{n+1}$ and $R(\lambda,\op{diag}(1,0)) = \Id_{n+1}$.  In other words, $R$ factors through a morphism (abusing notation slightly):
\[
R: \Gm{} \wedge \op{J}_n \longrightarrow \op{SL}_{n+1}.
\]
This morphism is an algebro-geometric avatar of Bott's generating complex; it is adjoint to a map $\op{J}_n \to \Omega^{1,1}\op{SL}_{n+1}$.

The description above can be made even more concrete in the case $n = 1$.  Indeed, a rank $1$ projection operator $\theta \in M_2(R)$ is characterized by the property that it has trace $1$ and determinant $0$.  Choosing coordinates $x_{ij}$ on $M_2(R)$, we then have $x_{11} + x_{22} = 1$ and $x_{11}x_{22} - x_{12}x_{21} = 0$.  In other words there is an isomorphism between $\op{J}_1$ and the smooth affine quadric $\op{Q}_2$ defined by the hypersurface $x_{12}x_{21} = x_{11}(1 - x_{11})$.  Our goal for the rest of this section is to establish the following result. After a preparatory digression, the proof can be found in \S\ref{sec:proof-of-equiv}.

	
	\begin{proposition}
		\label{prop:equiv-of-spheres}
		The morphism  
		\[
		R: \Gm \wedge \op{J}_{1} \longrightarrow \op{SL}_2 
		\]
		is a motivic equivalence.   
	\end{proposition}
	
	\subsection{Retracts of loop spaces}
	A pointed motivic space $\mathscr{X} \in \Spc(k)_*$ is a {\em retract of a loop space} if there is a pointed motivic space $\mathscr{Y}$ and maps $i: \mathscr{X} \to \Omega \mathscr{X}$, $r: \Omega \mathscr{Y} \to \mathscr{X}$ such that $r \circ i$ is homotopic to the identity.  Since $\Omega \mathscr{Y}$ is a module for the monad $\Omega \Sigma$, a map $\mathscr{X} \to \Omega \mathscr{Y}$ factors through a map $\Omega\Sigma \mathscr{X} \to \Omega \mathscr{Y}$.  Thus, a motivic space $\mathscr{X}$ is a retract of a loop space if and only if the unit map $\mathscr{X} \to \Omega \Sigma \mathscr{X}$ admits a retraction.  We use the following fact about equivalences between retracts of loop spaces.
	
	\begin{lem}
		\label{lem:equivalencesofretractsofloopspaces}
		Let $\mathscr X, \mathscr G_1, \mathscr G_2 \in \Spc(S)_*$. Assume $\mathscr G_1$ and $\mathscr{G}_2$ are retracts of loop spaces.  
		\begin{enumerate}[noitemsep,topsep=1pt]
			\item If $f_1,f_2: \mathscr{X} \to \mathscr{G}_1$ are two maps and if $\Sigma f_1$ is homotopic to $\Sigma f_2$, then $f_1$ is homotopic to $f_2$.	
			\item If $f: \mathscr{G}_1 \to \mathscr{G}_2$ is a pointed morphism (not necessarily compatible with the retractions), and $\Sigma f$ is an equivalence, then so is $f$.
		\end{enumerate}
	\end{lem}
	
	\begin{proof}
		(1)	Since $\mathscr{G}_1$ is a retract of a loop space, the map $\mathscr{G}_1 \to \Omega \Sigma \mathscr{G}_1$ admits a retraction.  The existence of such a retraction implies that
		\[
		[\mathscr{X},\mathscr{G}_1] \longrightarrow [\mathscr{X},\Omega \Sigma \mathscr{G}_1]  
		\]
		is a split monomorphism.  To check that two maps in the source agree, it therefore suffices to check that they agree in the target.  If $f_1$ and $f_2$ lie in the source, then to check their images in the target coincide, it suffices by adjunction to check that $\Sigma f_1 = \Sigma f_2$, as claimed. \newline
		
		\noindent (2) Consider the commutative diagram
		\[
		\xymatrix{
			\mathscr{G}_1 \ar[r]\ar[d]^{f} &  \Omega \Sigma \mathscr{G}_1 \ar[d]^{\Omega \Sigma f} \\
			\mathscr{G}_2 \ar[r] & \Omega \Sigma \mathscr{G}_2
		}
		\]
		Since $\Sigma f$ is an equivalence, so is the right hand vertical map.  In that case, there is a map $\tilde{g}: \mathscr{G}_2 \to \Omega \Sigma \mathscr{G}_1$ making both triangles commute (compose $\mathscr G_2 \to \Omega \Sigma \mathscr G_2$ with an inverse of $\Omega\Sigma f$).  We may then choose a retraction $r: \Omega \Sigma \mathscr{G}_1 \to \mathscr{G}_1$, and then set $g$ to be $r \circ \tilde{g}$.  We claim that $g$ is an inverse of $f$.  Note that commutativity of the upper triangle implies that $g \circ f$ is equivalent to $Id_{\mathscr{G}_1}$, so it remains to analyze the composite $f \circ g$.  By assumption $\Sigma f$ is an equivalence, so it follows that $\Sigma g$ is an inverse to $\Sigma f$.  In that case, since $\Sigma (f \circ g) = \Sigma f \circ \Sigma g$, we conclude from the first part of the lemma.
	\end{proof}
	\begin{rem}
		This result actually has very little to do with motivic spaces, and holds in fact for any adjunction between categories.
		(Specialize to the situation at hand via the adjunction $\Sigma: \Spc(S)_* \leftrightarrows \Spc(S)_*: \Omega$.)
	\end{rem}
	
	\subsection{Proof of Proposition~\ref{prop:equiv-of-spheres}}\label{sec:proof-of-equiv}
	\begin{proof}
		The map $R$ is defined over $\Z$; therefore by base-change it suffices to check the result over $\Z$.  Note that $\Gm{} \wedge \op{J}_1$ and $\op{SL}_2$ are both retracts of loop spaces.  Indeed, $\op{SL}_2$ can be identified with $\Omega B\op{SL}_2$: by appealing to \cite[Theorem 3.1.1]{asok2015affine} we may apply \cite[Theorem 2.2.5]{asok2015affine} to the fiber sequence $\SL{2} \to * \to B_\Nis \SL{2}$.  On the other hand, since $\op{J}_1$ is equivalent to ${\mathbb P}^1$ by appeal to Lemma~\ref{lem:JdeviceofPnproperties}, it follows that $\Gm{} \wedge \op{J}_1$ is equivalent to ${\mathbb A}^2 \setminus 0$, which in turn is equivalent to $\op{SL}_2$ under the map $\op{SL}_2 \to {\mathbb A}^2 \setminus 0$, given by projection onto any row or column.  By appeal to Lemma~\ref{lem:equivalencesofretractsofloopspaces} to check that the given map $R$ is an equivalence, it suffices to show that $\Sigma R$ is an equivalence.  
		
		In fact, it is possible to establish this equivalence in a completely geometric way, but we give a shorter approach here.  Indeed, to check $\Sigma R$ is an equivalence, it suffices to check that the map $R$ is an equivalence over prime fields by appeal to Proposition~\ref{prop:unstable-loc}.  In that case, since both spaces are equivalent to ${\mathbb A}^2 \setminus 0$, the map $R$ defines an element of $[{\mathbb A}^2 \setminus 0,{\mathbb A}^2 \setminus 0]$ which is isomorphic to $GW(k)$ by Morel's computations \cite[Theorem 7.15]{A1-alg-top}.  Therefore, suffices to prove that $a_k := [R] \in GW(k)$ is a unit.
		Now recall that for a finite field $k$, an element $a_k \in GW(k)$ is a unit if and only if its rank is $\pm 1$, whereas an element $a_\Q \in GW(\Q)$ is a unit if and only if its rank and signature are $\pm 1$.
		Let us note that both the real and complex realization of $R$ are equivalences\todo{ref?}.
		This immediately implies that the rank and signature of $a_\Q$ are units, as needed.
		Let $\ell \ne p$ be primes.
		Working over $\Z[1/\ell]$, the map $R$ induces an element \[ a_\ell := R^*(1) \in H^3_{et}((\Gm \wedge J^1)_{\Z[1/\ell]}, \Z_\ell(2)) \wequi \Z_\ell \]
		Base changing to $\F_p$ or $\C$ we see that \[ deg(a_{\F_p}) = a_\ell = deg(a_\Q) \in \{\pm 1\}, \] as needed.
	\end{proof}
	
	
\section{James constructions and localizations} \label{sec:james}
Throughout this section, we fix a presentably symmetric monoidal $\infty$-category $\mathscr C$ and a map $\alpha: \1 \to A \in \mathscr C$.
More generally we work with a map $\beta: S \to B \in \mathscr C$; the previous case is recovered by setting $S=\1$ and $B=A$.

\subsubsection*{$\alpha$-localization}
We call an object $X \in \mathscr C$ $\beta$-local if $\iMap(\beta, X)$ is an equivalence.
In other words this means that for $Y \in \mathscr C$ the canonical map \[ \beta^*: \Map(Y \otimes B, X) \longrightarrow \Map(Y \otimes S, X) \] is an equivalence.
We denote the full subcategory of $\beta$-local objects by $\mathscr C[\beta^{-1}]$. The inclusion admits a left adjoint, yielding a Bousfield localization \cite[\S5.5.4]{HTT} 
\[ 
\xymatrix{
	\mathscr C \ar@<.4ex>[r]^{L_{\beta}} & \ar@<.4ex>[l] \mathscr C[\beta^{-1}].} \]
The most common case we have in mind is $\beta = \alpha$, but we will occasionally also use the generalization.
For the convenience of the reader, we collect the well-known facts about these Bousfield localizations we shall use.

\begin{lem}
	\label{lem:rholocalizationproperties}
	\begin{enumerate}[noitemsep,topsep=1pt]
		\item If $\mathscr{X} \in \mathscr C$ is $\beta$-local and $\mathscr{Y} \in \mathscr C$, then $\iMap(\mathscr{Y},\mathscr{X})$ is $\beta$-local as well.  
		\item Any limit of $\beta$-local objects is again $\beta$-local.
		\item The category $\mathscr C[\beta^{-1}]$ acquires a canonical symmetric monoidal structure such that the functor $\mathrm{L}_{\beta}$ becomes symmetric monoidal.
	\end{enumerate}
\end{lem}

\begin{proof}
	(1) and (2) are immediate from the definitions.
	(3) follows from the construction of the localization and \cite[Proposition 4.1.7.4]{lurie-ha}.
\end{proof}

One can build $\mathrm{L}_{\alpha}$ by means of a small object argument, but it is useful to have a collection of spaces where $\alpha$-localization can be described in more elementary terms.  If $X \in \mathscr C$, then we can consider the telescope obtained by repeated application of $\alpha$:
\begin{equation} \label{eq:alpha-tel}
X[\alpha^{-1}] := \colim(X \stackrel{\alpha}{\longrightarrow} A \otimes X \stackrel{\alpha}{\longrightarrow} A \otimes A \otimes X \longrightarrow \cdots);
\end{equation}
we refer to $X[\alpha^{-1}]$ as the {\em $\alpha$-periodization} of $X$.  Later, we will give criterion under which $\mathrm{L}_{\alpha}$ is modelled by $X[\alpha^{-1}]$.

Let $\Alg(\mathscr C)$ denote the $\infty$-category of $\mathscr E_1$-monoids in $\mathscr C$.
Denote by 
\[ 
\xymatrix{\mathscr C \ar@<.4ex>[r]^-{F} & \ar@<.4ex>[l]^-{U} \Alg(\mathscr C)} 
\] the canonical adjunction.
We can $\beta$-localize algebras in $\mathscr C$, in the following sense.

\begin{cor} \label{cor:alphaloc-mon}
	There exists a unique Bousfield localization of $\Alg(\mathscr C)$ such that both $F$ and $U$ preserve $\beta$-local equivalences.
	Moreover the forgetful functor $U$ commutes with $L_\beta$.
	
	The same statement holds true with the associative operad replaced by any $\infty$-operad; in particular the analogous statement holds for the forgetful functors $\mathscr C_* \to \mathscr C$ and $\CAlg(\mathscr C) \to \mathscr C$.
\end{cor}
\begin{proof}
	This is the same argument as in \cite[Amplification 3.1.9]{ABHFreudenthal}.
\end{proof}

\subsubsection*{James construction}
\begin{defn} \label{def:james}
The forgetful functor $\Alg(\mathscr C) \to \mathscr C_{\1/}$ admits a left adjoint.
For $(\alpha: \1 \to A) \in \mathscr C_{\1/}$, we denote the result of
applying this left adjoint by $J_\alpha$.
\end{defn}
Note that by construction, the unit $\1 \to J_\alpha$ factors as $\1 \xrightarrow{\alpha} A \to J_\alpha$.

\begin{lem} \label{lemm:rho-equiv-J}
	For each $X \in \mathscr C$, the map $X \to J_\alpha \otimes X$ is an $\alpha$-equivalence.
\end{lem}
\begin{proof}
	We need only treat the case $X = \1$.
	For $X \in \mathscr C$ we set $X_+ = X \amalg \1 \in \mathscr C_{\1/}$; this is the left adjoint to the forgetful functor.
	Note that we have a pushout in $\mathscr C_{\1/}$
	\begin{equation*}
		\begin{CD}
			\1_+ @>>> A_+ \\
			@VVV       @VVV \\
			\1 @>>> A.
		\end{CD}
	\end{equation*}
	Applying $J_{(\ph)}$ and writing $F: \mathscr C \to \Alg(\mathscr C)$ for the free algebra functor, we obtain a pushout square in $\Alg(\mathscr C)$
	\begin{equation*}
		\begin{CD}
			F(\1) @>>> F(A) \\
			@VVV       @VVV \\
			\1 @>>> J_\alpha.
		\end{CD}
	\end{equation*}
	The top map is an $\alpha$-equivalence by Corollary \ref{cor:alphaloc-mon}, and hence so is the bottom map (which we can view in either $\mathscr C$ or $\Alg(\mathscr C)$, thanks to Corollary \ref{cor:alphaloc-mon}).
\end{proof}

\subsubsection*{Central maps}
\begin{defn}
	Let $S \in \mathscr C$.  We will say that a map $\alpha: \1 \to A$ is $S$-central if the two maps \[ \alpha \otimes \id_A \otimes \id_S, \id_A \otimes \alpha \otimes \id_S: A \otimes S \longrightarrow A \otimes A \otimes S \] are homotopic.
\end{defn}

\begin{proposition} \label{prop:central-loc}
	Suppose that $\alpha$ is $S$-central.
	\begin{enumerate}[noitemsep,topsep=1pt]
		\item The map $J_\alpha \otimes S \to J_\alpha \otimes A \otimes S$ is an equivalence.
		\item The functor $(\ph) \otimes J_\alpha \otimes S$ inverts $\alpha$-equivalences.
		\item For $X \in \mathscr C$, $J_\alpha \otimes X$ is $(\alpha \otimes S)$-local.
	\end{enumerate}
\end{proposition}
\begin{proof}
	(1) Writing $\mu: J_\alpha \otimes A \to J_\alpha \otimes J_\alpha \to J_\alpha$ for the map induced by multiplication, the composite 
	\[ J_\alpha \xrightarrow{\id \otimes \alpha} J_\alpha \otimes A \stackrel{\mu}{\longrightarrow} J_\alpha \] is the identity by construction.
	Hence we need only prove that \[ J_\alpha \otimes A \stackrel{\mu}{\longrightarrow} J_\alpha \xrightarrow{\id \otimes \alpha} J_\alpha \otimes A \] becomes an equivalence after $\otimes S$.
	Our assumption is that
	\[
	\xymatrix{
		A\ar[r]^{\text{id}} \ar[dr]_-{\alpha \otimes \id} &  A \ar[d]^-{\id \otimes \alpha} \\
		&A\otimes A
	}
	\]
	commutes after $\otimes S$, so the result follows by tensoring with $J_{\alpha}$ on the left, $S$ on the right, and multiplying.
	In more details, we obtain the following commutative diagram
	\[
	\xymatrix{
		J_\alpha \otimes A \otimes S \ar[r]^{\id}\ar[dr]_-{\alpha \otimes \id}  &
		J_\alpha \otimes A\otimes S \ar[r]^{\mu} \ar[d]^{\id \otimes \alpha}& J_\alpha \otimes S \ar[d]^{\id \otimes \alpha}
		& 
		\\
		& J_\alpha \otimes A\otimes A \otimes S  \ar[r]^{\mu \otimes \id}        &
		J_\alpha \otimes A\otimes S, &
	}
	\]
	where in labelling the maps we have suppressed identity factors on the outer $J_\alpha$ and $S$.
	Since the path via bottom left is the identity, the result follows.
	
	(2)
	This holds since the functor preserves colimits and inverts generating $\alpha$-equivalences, by (1).
	
	(3) We must check that $\iMap_{\mathscr C}(\alpha \otimes S, J_\alpha \otimes X)$ is an equivalence.
	But $J_\alpha \otimes X$ is a module over $J_\alpha$, so this is the same as $\iMap_{\mathrm{Mod}_{J_\alpha}}(\alpha \otimes S \otimes J_\alpha, J_\alpha \otimes X)$.
	This is an equivalence since $\alpha \otimes S \otimes J_\alpha$ is, by (2).
\end{proof}

\begin{proposition} \label{prop:J-alpha-tel}
	Suppose that $\alpha$ is $S$-central.
	Suppose furthermore that the subcategory of $\mathscr C$ generated under filtered colimits by objects of the form $S \otimes X$ (for $X \in \mathscr C$) admits a functor to a $1$-category, which is conservative and preserves filtered colimits.\NB{is this really necessary?}
	Then $J_\alpha \otimes S \wequi S[\alpha^{-1}]$.
\end{proposition}
\begin{proof}
	We shall first prove that $\alpha \otimes \1[\alpha^{-1}]$ is an equivalence.
	To do so, write \[ A^{\otimes \bullet} = (\1 \xrightarrow{\alpha} A \xrightarrow{\id \otimes \alpha} A \otimes A \xrightarrow{\id \otimes \id \otimes \alpha} \dots) \] for the ind-object with colimit $\1[\alpha^{-1}]$.
	The map of ind-objects $\alpha: A^{\otimes \bullet} \to A^{\otimes 1+\bullet}$ consists of various maps $\alpha_n: A^{\otimes n} \to A^{\otimes 1+n}$ (induced by inserting $\alpha$ in an appropriate spot), together with homotopies making certain squares commute.
	There is another map of ind-objects $\alpha': A^{\otimes \bullet} \to A^{\otimes 1+\bullet}$ given just by shifting.
	We can similarly describe it as consisting of various maps $\alpha'_n$ and homotopies making various squares commute.
	The centrality assumption implies that $\alpha_n$ is homotopic to $\alpha_n'$.
	It follows that when applying a functor $F$ to a $1$-category, we get $F(\alpha) \wequi F(\alpha')$.\footnote{However, without further assumptions, it does not follow that $\alpha \wequi \alpha'$, since the homotopies making the squares commute may be different.}
	Since the latter is an isomorphism of ind-objects, so is the former.
	Taking $F$ to be our conservative functor preserving filtered colimits, we get the desired result.

	From this we deduce that the functor $(\ph) \otimes S[\alpha^{-1}]$ inverts $\alpha$-equivalences, and hence (applying the functor to $\1 \to J_\alpha$, which is an $\alpha$-equivalence by Lemma \ref{lemm:rho-equiv-J}) we see that $S[\alpha^{-1}] \wequi J_\alpha \otimes S[\alpha^{-1}]$.
	On the other hand, $\1 \to \1[\alpha^{-1}]$ is an $\alpha$-equivalence, which is inverted by $\otimes J_\alpha \otimes S$ (Proposition \ref{prop:central-loc}(2)), whence \[ J_\alpha \otimes S \wequi S[\alpha^{-1}] \otimes J_\alpha, \]
as required.
\end{proof}

\begin{ex} \label{ex:J-alpha-tel}
Taking $\mathscr C = \Spc_*$ and $S=S^1$, the subcategory generated under colimits by objects of the form $S^1 \wedge X$ is just $\Spc_{*,>0}$.
This has a functor to a $1$-category which is conservative and preserves filtered colimits, namely, $\pi_*$.
\end{ex}

\begin{rem}
	We could have defined $J_\alpha$ to be the free $\mathscr E_\infty$-monoid with unit $\alpha$ (instead of just asking for an $\mathscr E_1$-monoid) and no arguments in this section would be changed.
	Moreover Proposition \ref{prop:J-alpha-tel} shows that after $\otimes S$, the two possibilities coincide.
\end{rem}

\section{Some recollections abut the real-\'etale topology}
\label{sec:small-ret}
In this section, we collect various properties of real-\'etale topoi and unstable motivic homotopy theory in the real-\'etale topology.  

\subsubsection*{Real spectra}
If $A$ is a commutative ring, then its real spectrum $R\Spec A$ is the set of pairs $({\mathfrak p},\alpha)$ consisting of a prime ideal ${\mathfrak p} \subset A$ and an ordering $\alpha$ of the residue field (the notation $\op{Sper} A$ is often used in the literature).  One equips $R\Spec A$ with the topology generated by the subbasis consisting of positive principal open sets: given an element $r \in A$, we can consider those points of $R\Spec A$ for which $r >_{\alpha} 0$.  The topological space $R\Spec A$ is known to be spectral \cite[\href{https://stacks.math.columbia.edu/tag/08YF}{Tag 08YF}]{stacks-project} \cite[Remark 7.1.17]{BCR}.  

If $X$ is a scheme, then we write $RX$ for the attached real space, which is obtained by gluing real spectra (the notation $X_r$ is frequently used for this space in the literature).  The space $RX$ is {\em locally spectral} \cite[0.6.3]{real-and-etale-cohomology}.

\begin{proposition}
	\label{prop:coherence}
	If $X$ is a qcqs scheme, then $RX$ is coherent in the sense that the set of quasi-compact open subsets of $RX$ is closed under finite intersections and forms a basis for the topology of $RX$.
\end{proposition}
\begin{proof}
Immediate from the preceding observations.
\end{proof}

The points of $RX$ have a convenient description: they are equivalence classes of morphisms $\Spec r \to X$ where $r$ is a real closed field, where two morphisms $\Spec r \to X$ and $\Spec r' \to X$ are equivalent if we can find a common real closed subfield $r'' \subset r$, $r'' \subset r'$ such that the given morphisms factor through a morphism $\Spec r'' \to X$.

We write $\Shv(RX)$ for the $\infty$-topos of sheaves of spaces on $RX$.  The assignment $X \mapsto RX$ is functorial.  The function $supp: R\Spec A \to \Spec A$ defined by forgetting the ordering extends to a locally spectral map of locally spectral topological spaces $supp: RX \to X$.  This is in fact a natural transformation.

\begin{proposition}[{\cite[Proposition 1.7]{real-and-etale-cohomology}}]
	\label{prop:Rfiberproducts}
	Assume $\varphi: S' \to S$ is a morphism of schemes such that $\kappa(s')$ is algebraic over $\kappa(\varphi(s'))$ for every $s' \in S'$.  Given any morphism $f: X \to S$, the natural map
	\[
	R(X \times_S S') \longrightarrow RX \times_{RS} RS',
	\]
	where the fiber product on the right-hand side is formed in the category of topological spaces, is a homeomorphism.
\end{proposition}

\begin{ex}
	\label{ex:realfibers}
Suppose $S$ is a scheme and $(s,\alpha) \in RS$.  Write $r$ for the real closure of $\kappa(s)$ with respect to $\alpha$. Given any morphism $f: X \to S$, the above result implies that $R(\Spec(r) \times_S X)$ is homeomorphic to the fiber of $RX$ over $(s,\alpha)$.  We will freely use this in the sequel.
\end{ex}

\begin{ex}
	\label{ex:cechnerve}
	If $u: X' \to X$ is an \'etale morphism, and $C(u)$ is the Čech nerve of $u$, then $R(C(u)) = C(R(u))$ by repeated application of Proposition~\ref{prop:Rfiberproducts}. 
\end{ex}

A family of morphisms $ \{\varphi_i: U_i \to X\}_{i \in I}$ is said to be {\em real surjective} if the induced maps $RU_i \to RX$ are jointly surjective.  Real surjective families are known to satisfy the axioms for the coverings of a pretopology \cite[\S 1.2]{real-and-etale-cohomology}.  The topology on $\Et(X)$ defined by real surjective families is called the {\em real-\'etale topology}; we write $X_{\ret}$ for the corresponding site.  We write $\Shv(X_{\ret})$ for the $\infty$-topos of sheaves of spaces on $X_{\ret}$.  If $f: X' \to X$ is an \'etale morphism, then Scheiderer shows that $Rf$ is a local homeomorphism \cite[Proposition 1.8]{real-and-etale-cohomology}.  

While it is difficult to directly compare the sheaves on $RX$ and real-\'etale sheaves, the preceding observation allows a comparison.  Scheiderer builds an intermediate site and establishes that topoi of sheaves of sets on $RX$ and sheaves of sets on $X_{\ret}$ are naturally equivalent \cite[Theorem 1.3]{real-and-etale-cohomology}.  Carral--Coste \cite{CarralCoste} showed that the cohomological dimension of a spectral space is bounded by its Krull dimension.  Scheiderer extended this result to locally spectral spaces \cite{SchedeiderCD} (see also \cite[\href{https://stacks.math.columbia.edu/tag/0A3G}{Tag 0A3G}]{stacks-project}) and one thus obtains bounds on the cohomological dimensions of $X_{\ret}$.  

Elmanto--Shah promoted Scheiderer's comparison to a natural equivalence of $\infty$-topoi.  In conjunction with the fact that $RX$ is locally spectral, they also deduced convergence of Postnikov towers when $X$ has finite Krull dimension by combining the $\infty$-topos enhancement of Scheiderer with a variant of Scheiderer's cohomological dimension bound established by Clausen-Mathew \cite[Theorem 3.14]{ClausenMathew}.  We combine the above observations in the following result.

\begin{theorem}[{\cite[Theorems B.10 and B.13]{elmanto2021scheiderer}}]
	\label{thm:rethypercomplete}
	Suppose $X$ is a scheme.
	\begin{enumerate}[noitemsep,topsep=1pt]
	\item There is a natural equivalence of $\infty$-topoi of the form:
	\[
	\Shv(RX) \simeq \Shv(X_{\ret}).
	\]
	\item If $X$ has locally finite Krull dimension, then $\Shv(X_{\ret})$ is hypercomplete.
	\end{enumerate}
\end{theorem}

\begin{cor}
	\label{cor:compactgeneration}
	If $X$ is a qcqs scheme, then $\Shv(RX)$ and $\Shv(X_{\ret})$ are compactly generated, i.e., generated under filtered colimits by compact objects.
\end{cor}

\begin{proof}
	By \cite[Proposition 6.5.4.4]{HTT}, if $Z$ is a topological space that is coherent in the sense that quasi-compact opens are stable under finite intersections and form a basis for the topology of $Z$, then $\Shv(Z)$ is compactly generated.  The space $RX$ is coherent in this sense if $X$ is qcqs by Proposition~\ref{prop:coherence}, so we conclude $\Shv(RX)$ is compactly generated.  Theorem~\ref{thm:rethypercomplete}(1) then guarantees that $\Shv(X_{\ret})$ is compactly generated as well.
\end{proof}

For later use, we also recall some facts about points and specialization in the real-\'etale topology.  By a {\em real closed point} of a scheme $X$, we will mean a closed point with real closed residue field.  A ring $A$ is called {\em real henselian} if it is henselian local with real closed residue field and a {\em real closed valuation ring} if it is a valuation ring whose fraction and residue fields are real closed.  Real henselian rings are stalks in the real-\'etale topology \cite[\S 3.7]{real-and-etale-cohomology}.  If $A$ is a real closed valuation ring and $S$ is a scheme, then the morphism $\Spec A \to S$ determines a specialization in $RS$ and every specialization $x \rightsquigarrow y$ in $RS$ is obtained in this fashion \cite[(1.5.2)]{real-and-etale-cohomology}.

\subsubsection*{Comparing big and small sites}
Consider the inclusion $\iota: \Et_S \hookrightarrow \Sm_S$.  This functor induces a restriction functor at the level of presheaf categories which has a left adjoint ``extension" functor:
\[
\begin{tikzcd}
\mathrm{P}(\Et_S) \arrow[r, "\iota^p" ,shift left] & \arrow[l, "\iota_*", shift left]\mathrm{P}(\Sm_S);
\end{tikzcd}
\]
the functor $\iota^p$ is a left Kan extension.  

Consider the big site $(\Sm_S)_{\ret}$.  The functor $\iota: \Et_S \longrightarrow \Sm_S$ preserves covers by construction and also preserves fiber products, so it is a continuous morphism of sites in the sense of \cite[\href{https://stacks.math.columbia.edu/tag/00WU}{Tag 00WU}]{stacks-project}.  As a continuous morphism of sites, we conclude that restriction yields a functor 
\[
\iota_*: \Shv_{\ret}(\Sm_S) \longrightarrow \Shv(S_{\ret}),
\]
with left adjoint $\iota^*:= a_{\ret}\iota^p$.  Furthermore, $\iota_*$ commutes with formation of the associated sheaf (since \'etale morphisms are stable under composition, every cover of $X \in S_{\ret}$ in $\Sm_S$ comes from a cover in $S_{\ret}$) and thus preserves colimits.  In addition, $\iota^*$ preserves representable sheaves.

\begin{proposition}
	\label{prop:bigsmallcomparison}
	The inclusion functor $\iota: \Et_S \hookrightarrow \Sm_S$ induces a continuous morphism of sites $S_{\ret} \to (\Sm_S)_{\ret}$.  There is an induced adjunction
	\[
	\begin{tikzcd}
\Shv(S_{\ret}) \arrow[r, "\iota^*" ,shift left] & \arrow[l, "\iota_*", shift left]\Shv_{\ret}(\Sm_S),
	\end{tikzcd}
	\]
	where $\iota_*$ is colimit preserving and $\iota^*$ is fully-faithful.
\end{proposition}
\begin{proof}
We have already seen that $\iota_*$ preserves colimits.
If $h_X \in \Shv(S_\ret)$ is a representable sheaf, then $\iota_*\iota^* h_X = h_X$ (since this holds for presheaves, and both $i^* = a_\ret i^p$, $i_* a_\ret = a_\ret i_*$).  Since both $\iota_*$ and $\iota^*$ preserve colimits and any sheaf may be written as a colimit of representables, we conclude that $\iota^* \iota_* \mathscr{F} \simeq \mathscr{F}$ for any sheaf.  In other words, $\iota_*$ is fully-faithful.
\end{proof}

We aim to analyze the essential image of $\iota^*$.  To do this, we combine Proposition~\ref{prop:bigsmallcomparison} and Theorem~\ref{thm:rethypercomplete} to obtain the following result.

\begin{cor}
	\label{cor:realsheavesvsbigretsheaves}
	There is a colimit preserving, fully-faithful functor
	\[
	e: \Shv(RS) \longrightarrow \Shv_{\ret}(\Sm_S).
	\]
\end{cor}

\subsubsection*{Some real-\'etale sheaves}
Suppose $A$ is a commutative ring and $a\in A$.
We will say that $a \geq 0$ if for every real closed field $r$ and every ring homomorphism $\alpha:  A \to r$, $\alpha(a) \geq 0$; the assertion that $a$ satisfies some other inequality is defined similarly.

\begin{defn}
	\label{defn:ineq}
Define subpresheaves
\[
{\mathscr O}_{\geq 0}(A) := \{ a \in A  | a \geq 0\};
\]
We define subpresheaves $\mathscr{O}_{> 0}$, $\mathscr{O}_{\leq 0}$ and $\mathscr{O}_{< 0}$ and $\mathscr{O}_{[0,1]}$ in an analogous fashion.
\end{defn}

\begin{lem}
	\label{lem:gm}
	The canonical map $\mathscr{O}_{< 0} \sqcup \mathscr{O}_{> 0} \to \mathscr{O}^{\times}$ is an isomorphism after real-\'etale sheafification.
\end{lem}

\begin{proof}
	We may check the isomorphism stalkwise (the sheaves being discrete).  Thus, assume $A$ is a real Henselian local ring with residue field $k$. Since $A$ admits a map to a real closed field (namely $k$), it follows that every $a \in A$ satisfies at most one of $a > 0$ or $a < 0$; hence the map is injective. Now, if the image of $a$ is positive in $k$, then $a$ is a square.  Since $a$ is a unit by assumption, it is a square unit, and therefore positive in any other ordered field.  A similar statement holds if $a$ is negative in $k$ and thus the map is also surjective.
\end{proof}



\begin{lem} \label{lemm:construct-sect}
	For any section $s \in a_\ret {\mathscr O}$, the section $1+s^2$ is invertible.
\end{lem}

\begin{proof}
	Indeed if $A$ is a real henselian ring then $1+s^2 \in A$ has positive image in the residue field, whence is a unit.
\end{proof}

\subsubsection*{Real étale homotopy sheaves}
\begin{lem} \label{lem:ret-sheaves-detect}  Let $S$ be a quasi-compact, quasi-separated scheme. 
	\begin{enumerate}[noitemsep,topsep=1pt]
		\item If $S$ is the spectrum of a semilocal ring and $F \in \Shv(S_\ret)$,  then $\pi_0 F(S) \wequi (a_\ret \pi_0 F)(S)$.
		\item If  $\tau \in \{Zar, Nis\}$ and $E$ a $\tau$-hypersheaf of spaces on $S$, then $E$ is a real étale hypersheaf if and only if $\ul\pi_i^\tau E$ is a real étale sheaf for every $i$ (and every $\tau$-local base point).
		\item If in addition $S$ has finite Krull dimension, the same holds with ``sheaf'' in place of ``hypersheaf''.
	\end{enumerate}
\end{lem}
\begin{proof}
	Point (1).  By Theorem~\ref{thm:rethypercomplete} we have a canonical equivalence $\Shv_\ret(S) \wequi \Shv(R(S))$.  Let $U \to R(S)$ be an open cover and let $V \to U \times_{R(S)} U$ be an open cover of the $2$-fold fiber product.  The diagram $V \rightrightarrows U$ can be extended to a hypercovering of $R(S)$; we refer to it as a $2$-stage hypercover.  The collection of all such $2$-stage hypercovers is filtered by refinement.  We may compute $(a_{\ret}\pi_0F)(S)$ by taking the filtered colimit over all $2$-stage hypercovers of $R(S)$ of equalizers of the form\NB{ref?}
	\[
	\colim_{V \rightrightarrows U \to R(S)}\op{Eq}(\pi_0(F(U) \rightrightarrows \pi_0F(V)).
	\]
	Since $R(S)$ is spectral it is quasi-compact and every open cover has a finite subcover.  Since $\pi_0F$ turns finite disjoint unions into products, it thus suffices to show that any open covering of $RS$ is refined by a finite cover by disjoint open subsets.  

	To produce such a refinement, we make short digression.  Let $i: M \hookrightarrow R(S)$ be the inclusion of the subspace of closed points.  The space $M$ is compact Hausdorff, and the inclusion $i$ admits a continuous retraction $p$ that is defined by sending any point to its unique closed specialization \cite[Proposition 7.1.25]{BCR}.  Furthermore, since $S$ is semi-local, $M$ is also totally disconnected by \cite[Lemma 19.2.2]{real-and-etale-cohomology}.  As a compact, totally disconnected space, $M$ has a basis of clopen sets (this is a straightforward exercise in general topology, but see, e.g., \cite[Proposition 3.1.7]{topological-groups-and-related-structures}).  
	
	Now, if $U \subset M$ is open, then $p^{-1}(U) \subset R(S)$ is the smallest open subset containing $i(U)$. In particular, if $U \subset R(S)$ is open, then $p^{-1}i^{-1}(U)$ is an open subset of $U$.  It follows that if $U = \coprod_{\alpha} U_{\alpha} \to R(S)$ is an open cover of $R(S)$, then $\coprod_{\alpha} p^{-1}i^{-1}U_{\alpha}$ is a finer covering and $\coprod_\alpha i^{-1} U_\alpha$ is a covering of $M$. By the conclusion of the previous paragraph, this covering of $M$ is refined by a finite cover by disjoint opens. Applying $p^{-1}$ yields the desired refinement of the original cover.

Point (2).  By appeal to the conclusion of Point (1), we see that for any real étale sheaf $E$ on $S$ we get $\ul\pi_i^\tau(E) = \ul\pi_i^\ret(E)$ (for any local base point).\NB{Since Zariski local rings have $\ret$-cohomological dimension $0$ \cite[Proposition 19.2.1]{real-and-etale-cohomology}, the same conclusion for sheaves of spectra is easier.}  In particular, the condition is necessary.
	
Conversely, let $E$ be a $\tau$-hypersheaf such that each $\ul\pi_i^\tau E$ is a real étale sheaf (for any local base point).  If we write $L_{\ret}E$ for the associated real étale hypersheaf of spaces, then $\ul\pi_i^\tau L_{\ret}E = \ul\pi_i^\ret L_{\ret}E = \ul\pi_i^\ret E = \ul\pi_i^\tau E$ (for any base point) by the above observation, and so the map $E \to L_{\ret}E$ induces an isomorphism on $\ul\pi_i^\tau$ (for every local base point). Since it is a map of $\tau$-hypersheaves, it is an equivalence.

Point (3). Immediate from (2), since the topos of $\tau$-sheaves on $S$ is hypercomplete \cite[Theorems 3.18 and 3.12]{ClausenMathew}\NB{correct ref?}.
\end{proof}

\begin{cor}
	\label{cor:unstableretconnectivity}
	Suppose $X \in \Shv_{\Nis}(S)$ satisfies real-\'etale descent. The sheaf $X$ is $n$-connected in $\Shv_{\ret}(S)$ if and only if it is $n$-connected in $\Shv_{\Nis}(S)$.
\end{cor}

\begin{proof}
	Since connectivity can be checked stalkwise, this follows from Lemma~\ref{lem:ret-sheaves-detect}(1).
\end{proof}

\section{Comparing real-\'etale homotopy theories} \label{sec:main1}
\subsubsection*{The real affine line}
Suppose $r$ is a real closed field.  We now recall some facts about the real spectrum of $r[x]$.  We follow the discussion of \cite[\S III.3]{KnebuschScheiderer}. The support map $supp: R\Spec r[x] \to \Spec r[x]$ has image consisting of real reduced prime ideals; these are either (a) maximal ideals with residue field $r$ or (b) the ideal $(0)$ with residue field $r(x)$.  In the first case, the only possible ordering is that on $r$. We call these the \emph{rational points of $R\aone_r$.}  In the latter case, we need to analyze orderings on $r(x)$; this is done in \cite[\S II.9]{KnebuschScheiderer}.  

Recall from \cite[Definition 2.9.3-4]{KnebuschScheiderer} that if $(M,\leq)$ is a totally ordered set, then a generalized Dedekind cut of $M$ is a pair $(L,U)$ of subsets of $M$ such that $L \cup U = M$ and $l < u$ for every $l \in L$ and $u \in U$. If $L$ is a non-empty proper subset (equivalently, both $L$ and $U$ are non-empty), then we call $(L,U)$ a proper Dedekind cut.  Finally, a Dedekind cut is called {\em free} if $L$ does not have a greatest element and $U$ does not have a least element.  

The orderings of $r(t)$ are in bijection with generalized Dedekind cuts of $r$ \cite[Corollary 2.9.6]{KnebuschScheiderer}.  To describe these cuts, recall that any ordering is determined by its cone of non-negative elements, which is what we specify in what follows.  The cuts $(r,\emptyset)$ and $(\emptyset, r)$ give rise to orderings $P_{\infty}$ and $P_{-\infty}$.  The cuts $((-\infty,c],(c,\infty))$ and $((-\infty,c),[c,\infty))$ correspond to orderings $P_{c+}$ and $P_{c-}$.  The remaining cuts correspond precisely to free Dedekind cuts \cite[Proposition 2.9.5]{KnebuschScheiderer}.

\begin{rem}
	The orderings $P_{\pm \infty}$ correspond to the sign of a function in the limit, the orderings $P_{x_\pm}$ correspond to the sign ``just to the left/right of $x$'', and the orderings $P_{\xi}$ correspond to the ``sign at $\xi$"; see the discussion of \cite[\S 2.9]{KnebuschScheiderer} for further explanations.
\end{rem}

Putting all of this together, one deduces the following result, most of which is contained in \cite[Example 3.3.14]{KnebuschScheiderer}.

\begin{proposition}
	\label{prop:pointsofa1}
If $r$ is a real closed field, then the space $R(\aone_r)$ has the following points 
\begin{enumerate}[noitemsep,topsep=1pt]
	\item A closed point $x$ for every $x \in r$, and
	\item two points $x_\pm$ that specialize to $x$; 
	\item closed points $\pm \infty$.
	\item a closed point $\xi$ for every free Dedekind cut $\xi$ of $r$.
\end{enumerate}
Moreover, the following statements hold:
\begin{itemize}[noitemsep,topsep=1pt]
	\item the space $R(\aone_r)$ has a natural linear ordering extending the ordering of $r$;
	\item if $r'/r$ is an extension of real closed fields, then the induced map $R(\aone_{r'}) \to R(\aone_r)$ is surjective and preserves the orders just described.
\end{itemize}
\end{proposition}

\begin{proof}
	The only statement here that is not contained in \cite[Example 3.3.14]{KnebuschScheiderer} is the last one.  We observed that points of $R(\aone_r)$ correspond to generalized Dedekind cuts of $r$.  If $r \subset r'$ is an inclusion of real closed fields, then any generalized Dedekind cut of $r'$ yields one of $r$ by restriction.  Thus, the statement about preserving linear orders is straightforward from the definitions given above (see the picture in \cite[Example 3.3.14]{KnebuschScheiderer}).  
	
	It remains to establish surjectivity.  To see this, it suffices to observe that any generalized Dedekind cut of $r$ can be extended to one of $r'$.  If $(L,U)$ is a generalized Dedekind cut of $r$, then define $L' = \{ x \in r' | x < U\}$ and $U' = \{x \in r' | x > L'\}$.  Now, if $x \notin L'$, then there exists $y \in U$ with $x \geq y$.  If $z \in L'$, then $z < y$ and thus $z < x$.  In other words, $x \in U'$.  We conclude that $r' = L' \cup U'$ and thus $(L',U')$ is a generalized Dedekind cut whose intersection with $r$ is $(L,U)$.
\end{proof}


Now, assume $A$ is a real closed valuation ring with residue field $k$ and fraction field $K$ (recall: both are real-closed by definition). If $a \in A$, we write $\bar{a}$ for its image in $k$.  Given an element $a \in A$, the morphism $A[t] \to A$ sending $t$ to $a$ induces a morphism $\Spec A \to \Spec A[t]$ and thus corresponds to a specialization in $R(\aone_A)$.  Also, via Proposition~\ref{prop:pointsofa1}(1,2), any $a \in A$ yields points $a, a_\pm \in R(\aone_K) \subset R(\aone_A)$.  Keeping this notation in mind, we have the following result.

\begin{lem} 
	\label{lemm:specialization-real-valn}
	Let $A$ be a real closed valuation ring, with residue field $k$ and fraction field $K$.  If $a \in A$, then the corresponding points $a, a_{\pm} \in R(\aone_K)$ described above specalize to $\bar{a} \in R(\aone_k)$; these are the only rational points of $R(\aone_K)$ that specialize into $R(\aone_k)$.
\end{lem}

\begin{proof}
	For $a \in A$ the induced map $R\Spec A \to R\aone_S$ exhibits the specialization from $a \in R\aone_K$ to $\bar a \in R\aone_k$; the first claim follows since $a_{\pm}$ specialize to $a$.  To see that no other rational points specialize into $R(\aone_k)$, it suffices to show that if $x \in K \setminus A$, then the corresponding point of $\aone_A$ does not specialize into $\aone_k$.
	Write $\mathfrak{m} \subset A$ for the maximal ideal, and ${\mathfrak p}_x \subset A[t]$ for the ideal corresponding to $x$.  Observe that $x^{-1} \in \mathfrak{m}$ and hence $A[t]/({\mathfrak m} + {\mathfrak p}_x) \wequi A[x]/{\mathfrak m}A[x] = 0$ (where, as usual, $A[x] \subset K$ denotes the subring generated by $A$ and $x$), as needed.
\end{proof}

\subsubsection*{Real local contractibility}
\begin{cor} 
\label{cor:tensor-of-real-closed} 
If $r'/r$ is an extension of real closed fields, and $X \in \Sch_r$, then the induced map $RX_{r'} \to RX$ is surjective.
\end{cor}

\begin{proof}
It suffices to show that the fibers of $RX_{r'} \to RX$ are non-empty, so we may assume that $X$ is the spectrum of a real closed field $r''$.  If $r'/r$ has transcendence degree $0$, then it is algebraic and the result follows from Proposition~\ref{prop:Rfiberproducts}.  If $r'/r$ has transcendence degree $1$, then let $t \in r'$ be transcendental over $r$.  In that case, $r'/r(t)$ is algebraic so Proposition~\ref{prop:Rfiberproducts} reduces us to showing that the ordering on $r(t)$ (induced by the inclusion $r(t) \subset r'$) extends to an ordering on $r''(t)$; this follows from the surjectivity assertion in Proposition~\ref{prop:pointsofa1}.  The case where $r'/r$ has finite transcendence degree reduces to this one by a straightforward induction.  In general, we can write $r'$ as an increasing union of real closed subfields that are of finite transcendence degree over $r$ and then $RX_{r'} \to RX$ is exhibited as an inverse limit along surjections (and thus surjective).
\end{proof}
	


Fix a scheme $X$.  Consider the projection morphism $Rp: R\aone_X \to RX$.  By Proposition~\ref{prop:pointsofa1}, the fiber over any real closed point of $X$ is isomorphic to $R\aone_r$ and thus has a natural ordering.  In particular, given two real closed points $x_1$ and $x_2$ in the fiber of $p$ over a given real closed point of $X$, if $x_1 < x_2$, then it makes sense to speak of the closed interval $[x_1,x_2]$.

Our goal here is to build special neighborhoods of open subsets of $X$ in $\aone_X$.  If $g \in \mathscr{O}(X)$, then the graph of $g$ provides a section of $p: \aone_X \to X$.  In that case, we will consider open subsets of $R(\aone_X)$ that contain tubular neighborhoods of the graph of $g$.  However, we may actually work slightly more generally: since \'etale morphisms are local homemorphisms \cite[Proposition 1.8]{real-and-etale-cohomology}, locally, $R(p)$ has additional sections beyond those that are simply sections of $p$.  For example, if $\pi: X' \to X$ is an \'etale morphism that maps a subset $V' \subset R(X')$ homeomorphically to $V \subset R(X)$, then any morphism $g': X' \to \aone_X$ yields a section $g$ of the restriction of $R(p)$ to $V$; we will refer to such a section of $R(p)$ as a {\em polynomial section} of $R(p)$.  With that in mind, we now define a class of open sets that are tubular neighborhoods of graphs of polynomial sections of the projection $R(p): R(\aone_{X}) \to R(X)$; we formalize this as follows
\begin{defn}
	\label{defn:intervalshaped}
	An open subset $U \subset R(\aone_X)$ is {\em interval-shaped} if there exist an open subset $V \subset R(X)$, and a polynomial section $g$ of $R(p)$ such that for every $x \in U$, the closed interval connecting $x$ and $gpx$ in the fiber over $px$ is contained in $U$.
\end{defn}

We shall see later that interval-shaped open subsets are homotopic to their images in $X$.  Before that, we observe that any open subset of $R(\aone_X)$ can be refined by (quasi-compact) interval-shaped open subsets.

\begin{proposition}
	\label{prop:intervalshapedcovers}
	Any open subset of $R(\aone_X)$ can be covered by (quasi-compact) interval-shaped open subsets.
\end{proposition}

\begin{proof}
	Suppose $U$ is an open subset of $R(\aone_X)$, choose a real closed point $x \in U$ and write $y$ for its image under $R(p)$.  We will build an interval shaped neighborhood of $x$.  Since $U \cap R(\aone_y)$ is open, it contains an open interval $J$ around $x$.  Let $z$ be a rational point of $J$.  We claim that, after possibly shrinking $U$ further, there exists a polynomial section $g: V \to U$ with $g(y) = z$.  To see this, we may replace $X$ by its real henselization $A$ at $y$.  The rational point $z$ corresponds to a map $\bar{g}: A[t] \to A/{\mathfrak m}$, where ${\mathfrak m}$ is the maximal ideal corresponding to $z$. The required polynomial section thus corresponds to a lift $g: A[t] \to A$ of $\bar{g}$, which clearly exists.  Then, there exists an affine real \'etale neighborhood $X' \to X$ of $y$ over which $g$ is defined.
	
	Since $X' \to X$ is an \'etale neighborhood, so is $\aone_{X'} \to \aone_{X}$.  As we observed above, \cite[Proposition 1.8]{real-and-etale-cohomology} then impplies that the morphisms $R(X') \to R(X)$ and $R(\aone_{X'}) \to R(\aone_X)$ are local homeomorphisms.  Therefore, shrinking $U$ and $V$ if necessary, and after replacing $U$ by a homeomorphic open subset in $R(\aone_{X'})$, and similarly for $V$, we may replace $X$ by $X'$ and assume we have an actual section $g$ of $p: \aone_{X} \to X$.  Moreover, since $X'$ is affine, we know $R(\aone_{X'})$ is spectral and thus has a basis of quasi-compact opens.  Therefore, we may even assume $U$ is quasi-compact.  
	
	Since the open set $U$ is quasi-compact, it is constructible \cite[Corollary 7.1.13]{BCR}.  The image of $U$ under the projection $\aone_{X} \to X$ is then a constructible subset of $R(X)$ \cite[Proposition II.1.9]{andradas2012constructible}.  Consider the subset $W = \{x \in R(\aone_{X}) | [x,gpx] \subset U\}$.  If we write $X = \Spec B$, then the constructible subsets of $R(\aone_X)$ are precisely those described by first order sentences in the language of ordered fields with parameters in $B[t]$ \cite[II.1.7]{andradas2012constructible}.  With that in mind, the subset $W$ is constructible.  It follows that $W$ is quasi-compact (again \cite[Corollary 7.1.13]{BCR}) and by \cite[Corollary 7.1.22]{BCR}, to check that $W$ is open, it suffices to check that it is stable under generalization in $U$.
	
	Suppose $w \in W$ has generalization $\tilde{w}$.  Choose a real closed valuation ring $A$ and a morphism $\Spec A \to \aone_{X}$ exhibiting the specialization $\tilde{w} \rightsquigarrow w$.  Let us write $K$ and $k$ for the fraction and residue fields of $A$. The composite $\Spec A \to \aone_{X} \to X$ determines a specialization $\tilde{b} \rightsquigarrow b$ in $X$, where $\tilde{b}$ and $b$ are the images of $\tilde{w}$ and $w$.  Pulling back $p: \aone_X\to X$ along the morphism $\Spec A \to \aone_{X} \to X$ we may assume that $X = \Spec A$, and that $w$ and $\tilde{w}$ are rational points. In fact there is an element $a \in A$ such that the image in $K$ corresponds to $\tilde w$ and the image in $k$ corresponds to $w$---in the sequel we just write $\tilde w$ for $a$.  The element $g$ lifts to $A$.  In that case, we consider the fiber $R(\aone_{\tilde{b}})$.  We have to show that if $\tilde{y}$ is a point lying between $\tilde{w}$ and $g(\tilde{b})$, then $\tilde{y} \in U$.
To do this, we appeal to Proposition~\ref{prop:pointsofa1} and just check each of the four listed cases separately. 
	
	\noindent {\bf (Case 1)}  In this case, the point $\tilde{y}$ is a rational point that lies between $\tilde{w}$ and $g(\tilde{b})$ by assumption.  Now, $K$ is a real closed field by assumption and $A \subset K$.  We abuse notation and write $\tilde{w}$ also for the element of $A$ corresponding to $\tilde{w}$ and similarly for $\tilde{y}$.  In that case, we have the inequalities $\tilde{w} < \tilde{y} < g$, or the reverse.   By the equivalent characterizations of \cite[II.3.2]{andradas2012constructible}, real closed valuation rings are convex, so the inequality above guarantees that $\tilde{y} \in A$ as well.  In that case, Lemma~\ref{lemm:specialization-real-valn} allows us to conclude that $\tilde{y}$ specializes to the closed point corresponding to $\bar{\tilde{y}}$.  In that case, we conclude that $\bar{\tilde{y}}$ lies between $\bar{\tilde{w}}$ and $\bar{g}$.  The former corresponds to $w$ by construction, and the latter corresponds to $g(b)$ as required.
	
	\noindent {\bf (Case 2)}  In this case, $\tilde{y}$ is one of the points $\tilde{y}_{\pm}$.  By Lemma~\ref{lemm:specialization-real-valn}, points of this type specialize to points of type considered in the previous case, which we have already addressed. 
	
	\noindent {\bf (Case 3)} Since $g(\tilde{b})$ and $\tilde{w}$ are both finite by assumption, $\tilde{y}$ cannot be of this type. 
	
	\noindent {\bf (Case 4)} Suppose $\tilde{y}$ corresponds to a free Dedekind cut.  Since $U$ is a quasi-compact open, $U \cap R(\aone_{\tilde{b}})$ is a union of open intervals; such a union contains the point corresponding to a Dedekind cut if and only if it contains all points of type (1) sufficiently close to the cut.  Once again, this follows from the case of points of type (1).
\end{proof}

\begin{proposition} 
	\label{prop:A1-inv}
If $X$ is a scheme, then $R(p)^*: \Shv(R(X))^\comp \to \Shv(R(\aone_X))^\comp$ is fully faithful.
\end{proposition}

\begin{proof}

Consider the functor $p^*: {\mathrm P}(RX) \to {\mathrm P}(R\aone_X)$.
If $U \subset R\aone_X$ is open such that $p(U) \subset RX$ is also open\NB{this should always be the case}, then $(p^*F)(U) \wequi F(pU)$.\NB{ref?}
We claim that if $F \in \Shv(RX)^\comp$ and $V \subset RX$ is quasi-compact open, then $(p^*F)(R\aone_V) \wequi (L^\comp p^*F)(R\aone_V)$ (where $L^\comp$ denotes the hypersheafification functor).
Since by the previous remark this is the same as $F(V)$, we deduce that $p_*L^\comp p^* F \wequi F$, proving what we want.
In fact, since $RX$ has a basis of opens coming from $X' \to X$ étale, we may assume (replacing $X$ by $X'$) that $V=RX$.
Let $U_\bullet \to R\aone_X$ be a hypercovering.
By Proposition~\ref{prop:intervalshapedcovers}, we can refine $U_\bullet$ by a hypercovering with interval-shaped entries.\NB{details?}
It will be enough to show that for such hypercoverings we have $(p^*F)(R\aone_X) \wequi \lim_\Delta (p^*F)(U_\bullet)$ (indeed $(L^\comp p^* F)(X)$ is obtained as a filtered colimit over all hypercoverings \cite[Theorem 8.6]{DHIhypercovers}\NB{better ref?}).
Since interval-shaped subsets have open images, we have $(p^*F)(U_\bullet) \wequi F(pU_\bullet)$.
It is thus enough to show that $|pU_\bullet| \to RX$ is an equivalence in $\Shv(RX)^\comp$.
We can test this stalkwise; thus let $i: \{x\} \hookrightarrow RX$.
Note that we have $i^* p(U) = p'(i'^*(U))$ for any $U \subset RX$ with $pU$ open.

Thus we may assume that $X=x$ is the spectrum of a real closed field.
Let $\mathscr C$ be the category of quasi-compact open subsets of $R\A^1_x$ (that is, finite disjoint unions of quasi-compact open intervals\NB{i.e. $[a,b]$ where $a \in \{x_+, -\infty\}$ and $b \in \{x_-, \infty\}$ for $x \in r$}).
The object $U_\bullet$ is a hypercovering of the terminal object of $\mathscr C$.
Consider the functor $S: \mathrm{P}(\mathscr C) \to \Spc$, given by left Kan extending the functor $\mathscr C \to \Spc$ which preserves finite coproducts and sends intervals to $*$.
We need to show that $S(|U_\bullet|) = *$.
We show more generally that $S$ factors through the localization $\mathrm{P}(\mathscr C) \to \Shv(R\A^1_x) \wequi \Shv(R\A^1_x)^\comp$ (this is indeed a localization, namely at (finite) coverings, since $R\A^1_x$ is spectral \cite[Theorem 7.3.5.2]{HTT}).
Any covering $U_\bullet \to U$ in $\mathscr C$ consists of finitely many intervals covering a disjoint union of finitely many intervals.
Let $\mathscr S \subset r \cup \{\pm \infty\}$ be the set of endpoints of all the finitely many intervals obtained by iterated intersection of this initial set.
There is an order-preserving embedding of $\mathscr S$ into $[0,1]$.
This yields an embedding of the set of open intervals in $R\aone_x$ with endpoints in $\mathscr S$ into the set of open intervals in $[0,1]$, which preserves coverings and intersections.
It is thus enough to prove that if $U'_\bullet \to U' \subset [0,1]$ is (the Čech nerve of) a covering of an open interval by finitely many open intervals, then $|S'U_\bullet| \wequi *$, where $S'$ is defined analogously to $S$.
Since intervals in $[0,1]$ are contractible, this follows from \cite[Remark A.3.8]{lurie-ha}.
\end{proof}

\subsubsection*{Motivic localization and the real-\'etale topology}
We now study $\aone$-invariant real étale sheaves on smooth schemes, that is, real-étale motivic spaces. 

\begin{defn}
	If $S$ is a scheme, then we write  $\Spc_\ret(S) \subset {\mathrm P}(\Sm_S)$ for the subcategory of presheaves which are $\aone$-invariant and satisfy real-étale descent; we refer to $\Spc_{\ret}(S)$ as the category of real-\'etale motivic spaces.  
\end{defn}


Let $\Spc(S)$ be the usual category of motivic spaces, i.e., the subscategory of $\mathrm{P}(\Sm_S)$ consisting of presheaves of spaces that are $\aone$-invariant and satisfy Nisnevich descent.  Write $\mathrm{L}_{mot}$ for the motivic localization functor, which is left adjoint to the inclusion $\Spc(S) \subset \mathrm{P}(\Sm_S)$. 

Likewise, $\Shv_{\ret}(\Sm_S) \subset \mathrm{P}(\Sm_S)$ and we write $L_{\ret}^\naive$ for the left adjoint to this functor.  The inclusion $\Spc_{\ret}(S) \subset \mathrm{P}(\Sm_S)$ also has a left adjoint $\mathrm{L}_{\ret}$.  As the real-\'etale topology is finer than the Nisnevich topology, we have
\[
\Spc_{\ret}(S) := \Spc(S) \cap \Shv_{\ret}(\Sm_S) \subset \Spc(S).
\]
A priori, there is no reason for the $L_{\ret}^\naive$ to preserve $\aone$-invariant presheaves and similarly no reason for $\mathrm{L}_{mot}$ to preserve real-\'etale sheaves.  As such, there is no simple formula for $\mathrm{L}_{\ret}$ for general $S$.

\begin{proposition}
	\label{prop:modelforlret}
	If $S$ is quasi-compact and quasi-separated, then 
	\[
	\mathrm{L}_{\ret} \wequi \colim (\mathrm{L}_{mot} \to (L^\naive_{\ret}\mathrm{L}_{mot}) \to \mathrm{L}_{mot}(L^\naive_{\ret}\mathrm{L}_{mot}) \to (L^\naive_{\ret}\mathrm{L}_{mot})^{\circ 2} \to \cdots)
	\]
\end{proposition}

\begin{proof}
	The categories of $\A^1$-invariant presheaves and of real étale sheaves are closed under filtered colimits in $\mathrm{P}(\Sm_S)$ (use Theorem \ref{thm:rethypercomplete}(1) and \cite[Theorem 7.3.5.2]{HTT} for the real étale case). Writing the colimit in two ways, it is seen to be both a colimit of $\aone$-invariant presheaves and of $\ret$-sheaves.
The result follows.
\end{proof}

\subsubsection*{Motivic real étale spaces} 
\label{sec:mot-ret}
\begin{lem} \label{lemm:contract-interval}
The presheaf $\mathscr O_{[0,1]} \subset {\mathscr O}$ (in the sense of Definition \ref{defn:ineq}) is contractible in $\Spc_\ret(S)$.
\end{lem}

\begin{proof}
We claim that there is a section $f \in (a_\ret {\mathscr O})(\aone)$ such that $f(0) = 1$, $f(1) = 0$ and for any real closed field $r$ and a homomorphism $\alpha: {\mathscr O}(\aone) \to r$ we have $0 \le \alpha(f) \le 1$.  Assuming this, the homotopy of multiplication by $f$ provides an $\aone$-homotopy contracting $\mathscr{O}_{[0,1]}$.
The desired section is $f = (2/(1+x^2)-1)^2$, which exists by Lemma \ref{lemm:construct-sect}.
The relevant inequalities are checked by elementary algebra.

\end{proof}

\begin{cor} \label{cor:rho-inv}
The map 
\[ 
\rho: S^0 \longrightarrow \Gm \in \Spc_\ret(S)_* 
\] is an equivalence.
\end{cor}

\begin{proof}
By Lemma \ref{lemm:contract-interval}, it suffices to prove that this is a homotopy equivalence for the interval $\mathscr{O}_{[0,1]}$.  By Lemma~\ref{lem:gm}, it suffices to show that $\mathscr{O}_{> 0}$ can be contracted using $\mathscr{O}_{[0,1]}$.  This can be done using the standard homotopy $t(x-1) + 1$.\NB{is this really legit?}
\end{proof}

\begin{lem} \label{lemm:contract}
For $X \in \Sm_S$ consider the functor 
\[ 
e_X: \Shv(RX) \wequi \Shv_\ret(\Et_X) \longrightarrow \Shv_\ret(\Sm_S). 
\]
If $U \subset R\aone_X$ is interval shaped with image $V$ in $RX$, then $e_{\aone_X}(U) \stackrel{\aone}{\wequi} e_X(V)$.
\end{lem}

\begin{proof}
We begin by studying the functor $e_X$.  If $u: X' \to X$ is an étale morphism, we denote by $C(u) = X'^{\times_X (\bullet +1)}$ the Čech nerve.
Then by Example~\ref{ex:cechnerve} we have $R(C(u)) = C(R(u))$.  In particular (c.f. \cite[\S 6.2.3]{HTT}), 
\[ 
im(R(u)) \wequi |RC(u)| \in \Shv(X_\ret).
\]

Now, if $U \subset R(X)$ is specified by inequalities $\{f_1 > 0, \dots, f_n > 0\}$, then it is given by the image of 
\[ 
X' := X[t_1,\ldots,t_n,s_1,\ldots,s_n]/(t_i^2 - f_i,s_if_i -1) \longrightarrow X;
\]
the morphism $X' \to X$ is \'etale by construction.
We deduce that $e(U)$ is the image of $X' \to X \in \Shv_\ret(\Sm_S)$.

The functions ${\bf f} = (f_1,\ldots,f_n)$ define a morphism $X \to {\mathbb A}^n$, and there is a pullback square of the form
\[
\begin{tikzcd}
X' \arrow[r] \arrow[d]& \Gm^{\times n} \arrow[d] \\
X \arrow[r, "{\bf f}"] & {\mathbb A}^n,
\end{tikzcd}
\]
where the right-hand map is given by component-wise squaring.

On the other hand the image of the squaring map $\Gm \to \aone$ is real-étale locally given by $\mathscr{O}_{> 0}$: if $A$ is real henselian and $a \in A$, then $a$ is a square unit if and only it is positive in every real closed field to which $A$ maps.  All in all we have learned that $e$ takes subsets of $RX$ the form $\{f_1 > 0, \dots, f_n > 0\}$ to subsheaves of $X$ defined by the same inequalities (in the sense of Definition \ref{defn:ineq}).
Finite unions of such subsets can be obtained by using a finite disjoint union of schemes of the form $X'$, and infinite unions can be described via filtered colimits.
Since $e$ preserves colimits, it follows that is also sends arbitrary unions of subsets of $RX$ the previous form to analogous subsheaves of $X$.

With this preparation out of the way, we shall prove the actual result.
Recall that $U \subset R(\aone_X)$ is interval-shaped  (Definition~\ref{defn:intervalshaped}) if there exists an open subset $V \subset R(X)$ an \'etale morphism $X' \to X$ a morphism $g: X' \to \aone_X$ to which $U$ can be linearly contracted, and a lift $V \to R(X')$. To prove that $e(U) \stackrel{\A^1}{\wequi} e(V)$ we may as well replace $X'$ by $X$ (note that $V \times_{RX'} R\A^1_{X'} \wequi \A^1_V$ contains an isomorphic copy of $U$).
In that case, consider the homotopy
\[
\begin{split}
\aone \times \aone_{X} &\longrightarrow \aone_X \\
(t,x) &\longmapsto tx + (1-t)gp
\end{split}
\]
We claim that this restricts to 
\[ 
\mathscr{O}_{[0,1]} \times e_{\aone_X}(U) \longrightarrow e_{\aone_X}(U). 
\]
Assuming this, we conclude that $R(p): U \to V$ and $R(g): V \to U$ are inverse homotopy equivalences, as needed.  To establish this assertion, we unwind the definitions via the discussion of the first part of the proof.
The claim translates to the assertion that for every real closed field $r$ (over $S$), the map 
\[ 
\aone(r) \times (\aone_X)(r) \longrightarrow \aone_X(r) 
\] 
restricts to 
\[ 
\mathscr{O}_{[0,1]}(r) \times U(r) \longrightarrow U(r), 
\]
which holds by the definition of interval-shaped (Definition \ref{defn:intervalshaped}) and the description of the ordering on the affine line over a real-closed field (i.e., Proposition~\ref{prop:pointsofa1}).
\end{proof}

\begin{theorem} \label{thm:main}
If $S$ is a scheme, then the functor 
\[ 
\Shv(S_\ret) \longrightarrow \Spc_\ret(S) 
\] 
is an equivalence.
\end{theorem}

\begin{proof}
We first reduce to the case that $S$ is quasi-compact, quasi-separated and has finite Krull dimension. 
Indeed both sides are Zariski sheaves, so we may assume $S$ affine (so in particular qcqs).
To conclude, we observe that also both sides are continuous: they convert cofiltered limits of quasi-compact, quasi-separated morphisms with affine transition morphisms into colimits.
This yields the desired reduction, by writing $S$ as a cofiltered limit of finite type, affine $\Z$-schemes.
The continuity assertion itself is an $\infty$-topos-theoretic enhancement of \cite[Proposition 3.4.1]{real-and-etale-cohomology} (see also \cite[Proof of Theorem 4.2]{bachmann-SHet2}).

Thus from now on we assume $S$ qcqs of finite Krull dimension.
We next prove that the functor in question is fully faithful.  Full-faithfulness is equivalent to the assertion that given $F \in \Shv(S_\ret)$, the image $eF \in \Shv_\ret(\Sm_S)$ is $\aone$-invariant.  Since $S$ is quasi-compact, quasi-separated of finite Krull dimension, by appeal to Theorem~\ref{thm:rethypercomplete} the topoi are hypercomplete. In that case, the claim follows by appeal to Proposition~\ref{prop:A1-inv} (note that $(eF)(\A^1_X) \wequi (p_*p^*q^*F)(X)$, where we denote by $\A^1_X \xrightarrow{p} X \xrightarrow{q} S$ the structure morphisms).

The functor $\Shv(S_{\ret}) \to \Spc_{\ret}(S)$ preserves colimits by definition, thus to establish the result, it suffices to show that this functor has dense image.  To this end, suppose $X \in \Sm_S$ and $X' \to X$ is an étale morphism.  We establish that $X'$ is in the closure under colimits of the image of $e_S$ (up to $\aone$-homotopy), by induction on the dimension of $X$.  If $\dim X = 0$ then $X'$ is étale over $S$, and hence in the image of $e_S$.  If $\dim X > 0$, then Zariski locally on $X'$, we may find an étale map $X' \to \aone_Y$.  It suffices to prove that such $X'$ are in the closure under colimits of the image of $e_S$.

We shall in fact show that $X'$ is in the closure under colimits of the image of $e_Y$; by our inductive assumption this is contained in the closure under colimits of the image of $e_S$.  Using Proposition~\ref{prop:intervalshapedcovers} (and once again Theorem~\ref{thm:rethypercomplete}) we see that $\Shv(R(\aone_Y))$ is generated under colimits by interval shaped $U \subset R\aone_Y$.  Since $e_{\aone_Y}(U)$ is in the image of $e_Y$ by Lemma~\ref{lemm:contract}, we conclude. 
\end{proof}

\section{Motivic real-étale localization and $\rho$-localization} \label{sec:motivic-real-etale-locn}
In this section we show that real étale localization and $\rho$-localization coincide for connected motivic spaces over any base scheme.
After some preliminary lemmas about the situation over fields, our main result (Theorem \ref{thm:main-comparison-connected}) is proved at the end of this section.

\subsubsection*{$\rho$-localization and periodization}
Recall the map $S^0 \to \Gm$ over $\Spec {\mathbb Z}$ sending the non base-point of $S^0$ to the unit $-1$.  Abusing terminology, we will write $\rho$ for the corresponding map over any arbitrary base scheme $S$ obtained by extension of scalars.
We now use the terminology of the \S\ref{sec:james}, for $\alpha = \rho$ and $\mathscr C = \Spc(S)_*$.

There are a priori two notions of $\rho$-localization on $\Spc(S)_*$, corresponding to $\rho$ or $\rho_+$.
These coincide at least for connected spaces:
\begin{lem} \label{lemm:rho-loc-rho+-loc}
Let $X \in \Spc(S)_*$.
\begin{enumerate}[noitemsep,topsep=1pt]
\item If $X$ is $\rho_+$-local it is $\rho$-local.
  In particular every $\rho$-equivalence is a $\rho_+$-equivalence.
\item If $X$ is $\rho$-local and connected, it is $\rho_+$-local.
\end{enumerate}
The same is true with $\Sigma \rho$ in place of $\rho$.
\end{lem}
\begin{proof}
(1) The map $\rho$ is obtained from $\rho_+$ by forming a pushout in the arrow category (i.e. $A \wequi A_+ \amalg_{S^0} *$), whence it is a $\rho_+$-equivalence.
Both claims follow.

(2) Using again the natural pushout $A \wequi A_+ \amalg_{S^0} *$, we have the fiber sequence \[ \iMap_*(\ph, X) \to \iMap_*((\ph)_+, X) \to X. \]
It follows that for $X$ connected $\iMap_*(\rho, X)$ is an equivalence if and only if $\iMap_*((\rho)_+, X)$ is (use that $* \to X$ is an epimorphism of sheaves, so pullback along this map is conservative \cite[Lemma 6.2.3.16]{HTT}).

The statements about $\Sigma \rho$ are proved in the same way.
\end{proof}

There are similarly two notions of centrality, corresponding to the cartesian or smash product symmetric monoidal structures; they do \emph{not} coincide.
We will only ever use the one corresponding to smash product.
\begin{lem} \label{lem:rho-S1-central}
The map $\rho$ is $S^1$-central in $\Spc(S)_*$.
\end{lem}
\begin{proof}
Recall from Proposition \ref{prop:equiv-of-spheres} that the map \[ \Gm \wedge \op{J} \to \op{SL}_2, (\lambda, \theta) \mapsto (\lambda \theta + (\Id - \theta))\cdot \op{diag}(\lambda^{-1},1), \] is an equivalence, where $\op{J}$ is the Jouanalou device on $\P^1$, i.e. the scheme of rank $1$ projection operators.
Setting $\lambda={-1}$ we find that the map $\Sigma \rho \wedge \id$ is represented by \[ \op{J} \to \op{SL}_2, \theta \mapsto (\Id - 2\theta)\cdot H, \] where $H = \op{diag}(-1,1)$.
Composing with the flip map on $\Gm \wedge \Gm$, via Lemma \ref{lem:flip} below we find that $\Sigma \id \wedge \rho$ is represented by \[ \theta \mapsto (H (\Id - 2\theta) HH)^{-1}. \] Since both $H$ and $(\Id - 2\theta)$ square to the identity, this is the same as $\theta \mapsto (\Id - 2\theta)\cdot H$, as desired.
\end{proof}

\begin{lem} \label{lem:flip} \NB{I suspect there is a reference for this}
The suspension of the flip map on $\Gm \wedge \Gm$ corresponds to the automorphism of $\SL{2}$ given by $A \mapsto (\op{diag}(-1,1) \cdot A \cdot\op{diag}(-1,1))^{-1}$.
\end{lem}
\begin{proof}
Recall that $\Sigma \Gm \wedge \Gm \wequi \A^2 \setminus 0$.
We shall first consider the map flip map \[
\begin{split}F: \A^2 \setminus 0 &\longrightarrow \A^2 \setminus 0, \\ (x, y) &\longmapsto (y,x). 
	\end{split}\]
From the pushout square 
\[
\xymatrix{
\Gm\times \Gm  \ar[r]\ar[d]  & \Gm \times \A^{1}
\ar[d] \\
\A^{1}\times \Gm  \ar[r]        & \A^{2}-\{0 \}
}
\]
we learn that $F \wequi -T$, where $T$ is the map we are interested in: $F$ flips the two copies of $\Gm$ and reverses the suspension coordinate.

Recall also that the projection $\SL{2} \to \A^2 \setminus 0$ is an equivalence.
The endomorphism $F'$ of $\SL{2}$ sending 
\[
A = \begin{pmatrix}
a&b \\ c& d
\end{pmatrix}
\]
to 
\[
\begin{pmatrix}
b & -d \\
a & -c
\end{pmatrix} = \begin{pmatrix}
0 & 1 \\ 1 & 0
\end{pmatrix}\cdot A\cdot \diag(-1,1)
\]
intertwines the projection maps with $F$\NB{formulation?} and hence corresponds to $F$ under the equivalence.
Since $\SL{2}$ is $\A^1$-connected over $\Z$ \NB{ref?}, $F'$ is homotopic to the operation
\[
A\mapsto \diag(-1,1)\cdot A\cdot \diag(-1,1).
\]
Finally, by the Eckmann--Hilton argument, the group structure on
\[
[\SL{2},\SL{2}] \wequi [\Sigma \Gm\wedge \Gm, \SL{2}]
\]
coming from the fact that the domain is a suspension coincides with
the one coming from the group structure on $\SL{2}$.
It follows that $T$ (the map we are interested in) corresponds to $-F'$, which is obtained from $F'$ by composing with the group inversion operation.
This is the desired result.
\end{proof}

Recall $J_\rho$, the James construction on $\rho$, from Definition \ref{def:james}.
Recall also the (naive) $\rho$-periodization $X[\rho^{-1}]$ from \eqref{eq:alpha-tel}.
\begin{lem} \label{lem:rho-J-tel}
Let $X \in \Spc(S)_*$ be connected.
Then $J_\rho \wedge X \wequi X[\rho^{-1}]$, and this object is the $\Sigma\rho$-localization of $X$.
\end{lem}
\begin{proof}
For the first statement, since we can write $X$ as a colimit of suspensions (see e.g. Lemma \ref{lem:colim-of-susp} below), we reduce to the case $X = \Sigma Y$, which is treated (via Lemma \ref{lem:rho-S1-central}) in Proposition \ref{prop:J-alpha-tel} (note that, just as in Example \ref{ex:J-alpha-tel}, the subcategory of $\Spc(S)_*$ generated under colimits by suspensions is $\Spc(S)_{*,>0}$, and this admits a conservative, filtered colimit preserving functor to a $1$-category, namely $\ul\pi_*$).
For the second statement, we first show that $X \to X[\rho^{-1}]$ is a $\Sigma\rho$-equivalence, which again is seen to hold by writing $X$ as a colimit of suspensions.
Next we must prove that $J_\rho \wedge X$ is $\Sigma\rho$-local; this is Proposition \ref{prop:central-loc}(3).
\end{proof}

Since we used it above and will use it frequently in the sequel, we recall the following well-known fact.
\begin{lem} \label{lem:colim-of-susp}
Any pointed connected object in $\Spc(S)$ (or an $\infty$-topos) is a functorial colimit of suspensions.
\end{lem}
\begin{proof}
We first note that this result is true for the usual category $\Spc$ of spaces.
For example, by classical delooping theory\todo{ref}, if $X \in \Spc_*$ is connected, then the two-sided bar construction $|B(S^1, \Omega\Sigma, \Omega X)| \to X$ is an equivalence.
The case of presheaves on a small category follows immediately, since all relevant constructions are performed sectionwise.
Now let 
\[ \xymatrix{\mathrm{P}(\mathscr C) \ar@<.4ex>[r]^-{L} & \ar@<.4ex>[l]^-{R} \mathscr D}\] be an adjunction such that (1) for $X \in \mathscr D_*$ connected the canonical map $L(R(X)_{\ge 1}) \to LRX \to X$ is an equivalence and (2) $L$ preserves terminal objects (and hence pointed objects).
Then the result holds for $\mathscr D$: just write a connected $X \in \mathscr D_*$ as $L(R(X)_{\ge 1}) \wequi L(\colim (\Sigma \dots))$, using that we have already established the result for $\mathrm{P}(\mathscr C)$.
If $\mathscr D$ is an $\infty$-topos, then we may choose $L$ to be a left exact localization which thus satisfies (1) and (2).
If $\mathscr D$ is $\Spc(S)$, then we can set $\mathscr C = \Sm_S$; (2) holds trivially and (1) is easily verified, using that being connected in $\Spc(S)$ by definition means being connected as a Nisnevich sheaf.
\end{proof}

\begin{rem} \label{rmk:susp-conn}
Note that if $S$ is a scheme and $X \in \Spc(S)_*$, then $\Sigma X \in \Spc(S)_*$ \emph{is} connected.
This follows from \cite[\S2 Corollary 3.22]{A1-homotopy-theory}.
\end{rem}

\subsubsection*{Splitting real étale covers}
Throughout this section we work over fields.

\begin{lem} \label{lem:Jrho-loops}
If $k$ is a perfect field and $X \in \Spc(k)_*$ is simply connected, then $J_\rho \wedge \Omega X \wequi \Omega(J_\rho \wedge X)$.
\end{lem}
\begin{proof}
For $Y \in \Spc(k)$ pointed and connected, we know by Lemma \ref{lem:rho-J-tel} that $J_\rho \wedge Y \wequi L_{\Sigma \rho} Y$.
We can consider localization at $\Sigma\rho$ both in $\Spc(k)$ and $\Spc(k)_*$.
The $\Sigma\rho$-localization in $\Spc(k)$ is the same as the $(\Sigma\rho)_+$-localization in $\Spc(k)_*$, by Corollary \ref{cor:alphaloc-mon}.
For connected objects this is the same as the $\Sigma\rho$-localization in $\Spc(k)_*$, by Lemma \ref{lemm:rho-loc-rho+-loc}.
Moreover, the $\Sigma\rho$-localization in $\Spc(k)$ preserves finite products, by Lemma \ref{lem:rholocalizationproperties}(3).
All in all we conclude that the functor $J_\rho \wedge (\ph)$ preserves finite products of pointed, connected motivic spaces.

It follows that $J_\rho \wedge \Omega X$ is a monoid, and we have \[ B(J_\rho \wedge \Omega X) = |(J_\rho \wedge \Omega X)^{\times \bullet}| \wequi |J_\rho \wedge (\Omega X)^{\times \bullet}| \wequi J_\rho \wedge B\Omega X \wequi J_\rho \wedge X. \]
Taking loops on both sides concludes the proof (this is where we use that $k$ is perfect, so that $\Omega B M \wequi M$ for connected monoids $M$\NB{ref?}).
\end{proof}

\begin{cor} \label{cor:pi1-Jrho}
For any field $k$, $\ul\pi_1 \Sigma J_\rho \wequi a_\ret \Z$.
If $k$ is unorderable then $\Sigma J_\rho = *$.
\end{cor}
\begin{proof}
By essentially smooth base change\NB{details?} it suffices to prove the claims for prime fields, which are in particular perfect.
Now by Lemmas \ref{lem:Jrho-loops} and \ref{lem:rho-J-tel} we have \[ \Sigma J_\rho = J_\rho \wedge S^1 \wequi J_\rho \wedge \Omega BS^1 \wequi \Omega(J_\rho \wedge BS^1) \wequi \Omega(BS^1[\rho^{-1}]). \]
Hence \[ \ul\pi_1 \Sigma J_\rho \wequi \colim_i \ul\pi_1 (BS^1 \wedge \Gmp{i}). \]
We have (for $i>0$) $\ul\pi_1 (BS^1 \wedge \Gmp{i}) \wequi \ul{K}_i^{MW}$ by \cite[Corollary 6.43]{A1-alg-top}, and hence by \cite[Theorem 8.5]{jacobson-fundamental-ideal} (see also \cite[Corollary 19]{bachmann-real-etale}) we get $\ul\pi_1 \Sigma J_\rho \wequi a_\ret \Z$, as needed.

Now suppose $k$ is unorderable.
Note that since $\Sigma J_\rho$ is the $\Sigma \rho$-localization of $S^1$, \[ [\Sigma J_\rho, \Sigma J_\rho]_* \wequi [S^1,  \Sigma J_\rho]_* \wequi \ul\pi_1(\Sigma J_\rho)(k) = *. \]
Thus the identity of $\Sigma J_\rho$ is homotopic to the zero map, whence $\Sigma J_\rho = *$.
\end{proof}

\begin{lem} \label{lem:ret-transfers}
Let $k$ be a field of characteristic $0$ and $l/k$ be an étale algebra such that $\Spec(l) \to \Spec(k)$ is a real étale cover.
There exists a map $\tau: \Sigma J_\rho \to \Spec(l)_+ \wedge \Sigma J_\rho \in \Spc(k)_*$ which is a section of the canonical map $pr \wedge J_\rho: \Spec(l)_+ \wedge \Sigma J_\rho \to \Sigma J_\rho$.
\end{lem}
\begin{proof}
Choose a generator $f$ of $l/k$ (this is possible since $k$ has characteristic zero, using the primitive element theorem\NB{details/ref?}).
The generator $f$ induces a closed immersion $\Spec(l) \to \A^1_l$.
The minimal polynomial of $f$ yields a trivialization of the normal bundle.
We hence obtain \[ \tau^0: \P^1 \to \P^1/(\P^1 \setminus \Spec(l)) \wequi Th(N_{\Spec(l)/\P^1}) \wequi \P^1 \wedge \Spec(l). \]
Let $\chi \in\mathscr O(R \Spec(l))$ and write $pr: \Spec(l) \to \Spec(k)$ for the projection.
By Corollary \ref{cor:pi1-Jrho}, there is a corresponding map $\chi': pr^*(\Sigma J_\rho) \to pr^*(\Sigma J_\rho) \in \Spc(l)_*$.
Applying $pr_\sharp$ we obtain $(\cdot \chi): \Spec(l)_+ \wedge \Sigma J_\rho \to \Spec(l)_+ \wedge \Sigma J_\rho \in \Spc(k)$.
Set \[ \tau^\chi = (\cdot \chi) \circ (\tau^0 \wedge J_\rho). \]
We claim that for an appropriate choice of $\chi$, $\tau^\chi$ is the required section.
(Note that $\P^1 \wedge J_\rho \wequi S^1 \wedge J_\rho$.)
By Corollary \ref{cor:pi1-Jrho}, the composite $(pr \wedge J_\rho) \circ \tau^0$ corresponds to a function on $R\Spec(k)$ which we want to arrange to be the identity.
It will suffice to prove that the map $\mathscr O(R \Spec(l)) \to \mathscr O(R \Spec(k)), \chi \mapsto (pr \wedge J_\rho) \circ \tau^\chi$ is surjective.\footnote{In fact, a more careful analysis shows that taking $\chi$ to be the indicator function of the largest root in the fiber yields $\tau^\chi$ with the required property.}
To see this, note that up to multiplication by a unit in $\mathscr O(R \Spec(l)$ (corresponding to picking a different trivialization of the normal bundle), our map is the usual transfer; this was proved to be surjective in \cite[Corollary 21]{bachmann-real-etale}.
\end{proof}

\begin{rem}[Unreduced suspensions] \label{rmk:unreduced-susp}
Even if $X \in \Spc(S)$ does not admit a base point, the usual pushout defining $\Sigma X$ still makes sense and is connected, hence admits a unique base point up to homotopy.
In fact there are two obvious base points (the two tips of the cones).
Picking one of them to be the base point lifts $\Sigma X$ to $\Spc(S)_*$, and the other possible base point then defines a map $S^0 \to \Sigma X \in \Spc(S)_*$.
In fact this yields a pushout diagram (in $\Spc(S)_*$)
\begin{equation*}
\begin{CD}
S^0 @>>> \Sigma X \\
@VVV        @VVV  \\
* @>>> \Sigma(X_+).
\end{CD}
\end{equation*}
\end{rem}

\begin{lem} \label{lem:ret-rho-fields}
If $\Spec(l) \to \Spec(k)$ is a real étale cover with Čech nerve $\Spec(l)^{\times \bullet+1}$, then
\[ 
\Sigma|\Spec(l)^{\times \bullet+1}| \wedge J_\rho = *. 
\]
\end{lem}

\begin{proof}
If $k$ has positive characteristic then $J_\rho = *$ by Corollary \ref{cor:pi1-Jrho}.
We may thus assume that $k$ has characteristic $0$.
Set $N=\Sigma|\Spec(l)^{\times \bullet+1}| \wedge J_\rho$.
It suffices to prove that for $X \in \Spc(k)_*$ we have $[\Sigma X, N]_* = *$ (indeed $N$ is pointed and connected by Remark \ref{rmk:susp-conn}, so it suffices to show that $\ul\pi_i(N) = *$ for $i>0$, which follows from the condition).
Since $N$ is $\Sigma \rho$-local by Lemma \ref{lem:rho-J-tel}, we get $[\Sigma X, N]_* \wequi [X \wedge \Sigma J_\rho, N]_*$.
Lemma \ref{lem:ret-transfers} shows that this group injects into $[\Spec(l)_+ \wedge X \wedge \Sigma J_\rho, N]_*$.
Thus we may base change to $l$, at which point $|\Spec(l)^{\times \bullet+1}|$ becomes contractible and the result follows.
\end{proof}

\subsubsection*{$\rho$-periodization and real étale descent}
We now get back to the case of an arbitrary base.
\begin{theorem} \label{thm:main-comparison-connected}
Let $S$ be a scheme and $X \in \Spc(S)_*$ be connected.
Then \[ X[\rho^{-1}] \wequi X \wedge J_\rho \wequi L_\ret X. \]
This object is the $\rho$-localization of $X$.
\end{theorem}
\begin{proof}
The first equivalence is Lemma \ref{lem:rho-J-tel}.
It remains to show that for $X \in \Spc(S)_*$ connected, $X$ is $\Sigma\rho$-local if and only if it is real étale local, if and only if it is $\rho$-local.
All three conditions are Zariski local, so we may assume $S$ affine.
Write $\mathscr C(S) \subset \Spc(S)_*$ for the subcategory of objects that are connected and $\Sigma\rho$-local.
Using Lemma \ref{lem:continuity} below, we see that $\mathscr C$ converts cofiltered limits of diagrams of affine schemes into filtered colimits in $Pr^L$.
(This is true without the connectivity assumption, and the lemma allows us to pass to the subcategory of connected objects---which is generated under colimits by suspensions.)
A similar discussion applies to the other locality conditions.
Since any affine scheme is a cofiltered limit of finite type $\Z$-schemes, this shows that in proving the result we may assume $S$ affine of finite type over $\Z$, in particular, qcqs of finite Krull dimension.

With this simplifying assumption, let us revert back to having a connected space $X \in \Spc(S)_*$.
We shall now show that $\Omega(X \wedge J_\rho) =: E$ is a real étale sheaf.
Thus let $Y^\bullet \to Y \in \Sm_S$ be the Čech nerve of a real étale cover.
We wish to prove that $\Map(Y, E) \wequi \Map(|Y^\bullet|, E)$ (beware that these are spaces of maps between \emph{unpointed} objects---we ignore the base point of $E$).
Sine the base change $\Spc(S) \to \Spc(Y)$ preserves limits, colimits and is symmetric monoidal, for this we may assume $Y=S$.
We claim that $\Sigma|Y^\bullet| \wedge J_\rho = *$.
By Proposition \ref{prop:unstable-loc} we reduce to the case $S = \Spec(k)$, which is dealt with in Lemma \ref{lem:ret-rho-fields}.
The claim is proved.
Next observe that \[ \Map(|Y^\bullet|, E) \wequi \Map_*(|Y^\bullet|_+, E) \wequi \Map_*(\Sigma |Y_\bullet|_+, X \wedge J_\rho) \wequi \Map_*(\Sigma |Y_\bullet|_+ \wedge J_\rho, X \wedge J_\rho), \] using that $X \wedge J_\rho$ is $\Sigma\rho$-local and $\Sigma |Y_\bullet|_+ \wedge J_\rho$ is the $\Sigma\rho$-localization of $\Sigma |Y_\bullet|_+$, both by Lemma \ref{lem:rho-J-tel}.
Using the pushout from Remark \ref{rmk:unreduced-susp} together with $\Sigma|Y^\bullet| \wedge J_\rho = *$ we see that $\Sigma |Y_\bullet|_+ \wedge J_\rho \wequi \Sigma J_\rho$.
Hence reversing the previous logic we find \[ \Map_*(\Sigma |Y_\bullet|_+ \wedge J_\rho, X \wedge J_\rho) \wequi \Map_*(\Sigma J_\rho, X \wedge J_\rho) \wequi \Map_*(S^1, X \wedge J_\rho) \wequi \Map_*(S^0, \Omega(X \wedge J_\rho)) \wequi \Map(*, E), \] as needed.
This finishes the proof that $E=\Omega(X \wedge J_\rho)$ is a real étale sheaf.

Since $X$ is connected, from this it follows that $X \wedge J_\rho$ is real étale local (Lemma \ref{lem:ret-sheaves-detect}) and thus $\rho$-local (Corollary \ref{cor:rho-inv}).
Since $X \to J_\rho \wedge X$ is a $\rho$-equivalence (Lemma \ref{lemm:rho-equiv-J}), this concludes the proof.
\end{proof}

\begin{lem} \label{lem:continuity} \todo{surely there must be a reference for this?}
Let $I$ be a filtered category and $\mathscr C: I \to Pr^L$.
Suppose that each $\mathscr C(i)$ is compactly generated and for each $\alpha:i \to j \in I$ the induced functor $\alpha: \mathscr C(i) \to \mathscr C(j)$ preserves compact objects.
Suppose further we are provided for each $i$ with a subset $S(i) \subset \mathscr C(i)$ of compact objects.
Write $\mathscr C_0(i) \subset \mathscr C(i)$ for the closure under colimits of $S(i)$ and suppose that for $\alpha:i \to j \in I$ we have $\alpha(\mathscr C_0(i)) \subset \mathscr C_0(j)$.

Then $\colim_I \mathscr C_0 \to \colim_I \mathscr C$ identifies the source with the full subcategory of the target obtained as the closure under colimits of the images of all the $S(i)$.
\end{lem}
\begin{proof} \NB{this is a little bit sketchy}
Write $Sect(\mathscr C)$ for the category of sections of $\mathscr C$, that is, families of objects $X_i \in \mathscr C(i)$ together with maps $\alpha(X_i) \to X_j$ for every $\alpha: i \to j$.
Write $\alpha_*: \mathscr C(j) \to \mathscr C(i)$ for the right adjoint and $\omega_\alpha: X_i \to \alpha_* X_j$ for the adjoint map.
There is a localization $L$ of $Sect(\mathscr C)$ onto objects where each $\omega_\alpha$ is an equivalence, and $LSect(\mathscr C) \wequi \colim_I \mathscr C$ (combine \cite[Corollary 5.5.3.4, Proposition 5.5.3.13, Corollary 3.3.3.2]{HTT}.
For $(X_i)_i \in Sect(\mathscr C)$, let $(L_0 X)_i = \colim_{\alpha: i \to j} \alpha_* X_j$.
This upgrades to a functor $L_0: Sect(\mathscr C) \to Sect(\mathscr C)$.
Using that each $\alpha_*$ preserves filtered colimits, we see that $L_0 X$ is $L$-local, which in particular implies that the two canonical maps $L_0 \to L_0^2$ are equivalences.
It follows \cite[Proposition 5.2.7.4]{HTT} that $L_0$ is a localization, which is at least as strong as $L$.
But $L_0 L \wequi L$, so $L_0 = L$.

The same discussion applies to $Sect(\mathscr C_0)$.
Write $i: Sect(\mathscr C_0) \to Sect(\mathscr C)$ for the canonical functor, and $i_*$ for its right adjoint (which is computed sectionwise).
Then $i_*$ preserves filtered colimits (filtered colimits being computed sectionwise, and $\mathscr C_0(j) \to \mathscr C(j)$ preserving compact generators) and hence commutes with $L=L_0$.
From this we conclude that \[ i: LSect(\mathscr C_0) \wequi \colim_I \mathscr C_0 \to \colim_I \mathscr C \wequi Sect(\mathscr C) \] is fully faithful.
The generation under colimits is clear since this holds for $Sect(\mathscr C_0)$.
\end{proof}

\section{Real realizations} \label{sec:realizations}
In this section we collect some complementary results.
Some of them are independent of the rest of the article, but seem to fit thematically and do not seem to be recorded in the literature.

\subsection{Real étale localization of nilpotent motivic spaces}
We study the behavior of the functor $L_\ret: \Spc(k) \to \Spc(k)$ when $k$ is a field.
\begin{lem} \label{lemm:Lret-fiber}
If $F \to X \to B \in \Spc(k)$ is a fiber sequence, such that $B$ is connected and $\ul\pi_1 B$ a real étale sheaf, then $L_\ret F \to L_\ret X \to L_\ret B$ is also a fiber sequence.
\end{lem}
\begin{proof}
If $char(k) > 0$, then $L_\ret = 0$ and there is nothing to prove.  We may thus assume that $char(k) = 0$.  Recall the description of $\mathrm{L}_{\ret}$ from Proposition~\ref{prop:modelforlret}: $\mathrm{L}_{\ret}$ may be written as an infinite composite of $L^\naive_{\ret}$ and $\mathrm{L}_{mot}$.  It is clear that $L^\naive_{\ret}$ preserves fiber sequences.  On the other hand since $\ul\pi_1 B$ is already a real-\'etale sheaf, $L^\naive_{\ret}$ does not modify it, and hence neither does ${\mathrm L}_{mot}$ by \cite[Theorem 2.3.8]{AWW}.  Thus at any stage $\ul\pi_1 B$ is strongly $\aone$-invariant, and so ${\mathrm L}_{mot}$ preserves the fiber sequence by \cite[Theorem 2.3.3]{AWW}.
\end{proof}

\begin{lem} \label{lemm:Lret-conn}
If $X \in \Spc(k)$ is $n$-connected, then $L_\ret X$ is again $n$-connected.
\end{lem}

\begin{proof}
By appeal to Proposition~\ref{prop:modelforlret} it again suffices to prove this for ${\mathrm L}_{mot}$ and $L^\naive_{\ret}$, and we may again assume that $char(k) = 0$.  That $\mathrm{L}_{mot}$ preserves connectivity is precisely Morel's unstable connectivity theorem, i.e., \cite[Theorem 6.38]{A1-alg-top}.  That $L^\naive_{\ret}$ preserves connectivity follows from Corollary~\ref{cor:unstableretconnectivity}.
\end{proof}

\begin{lem} \label{lemm:real-et-loc-approx}
If $X \in \Spc(k)_*$ is connected, and $\ul\pi_i X$ is a real étale sheaf for $i \le n$, then $X \to L_\ret X$ induces an isomorphism on $\ul\pi_i$ for $i \le n$.
\end{lem}

\begin{proof}
The fiber sequence $F \to X \to X_{\le n}$ has $F$ $n$-connected and induces a fiber sequence $L_\ret F \to L_\ret X \to L_\ret X_{\le n}$ by appeal to Lemma \ref{lemm:Lret-fiber}.  Now $L_\ret(X_{\le n}) \wequi X_{\le n}$ by Lemma \ref{lem:ret-sheaves-detect}(2), and $L_\ret F$ is still $n$-connected (Lemma \ref{lemm:Lret-conn}), so $L_\ret(X)_{\le n} \wequi X_{\le n}$ as needed.
\end{proof}

Recall the notion of a nilpotent motivic space from \cite{AFHLocalization}.
\begin{lem}
	\label{lem:lretpreservesnilpotence}
If $X \in \Spc(k)$ is nilpotent, then $L_\ret X$ is also nilpotent.
\end{lem}

\begin{proof}
Since the Postnikov tower of $X$ admits a principal refinement \cite[Theorem 3.1.13]{AFHLocalization}, there are fiber sequences 
\[ 
X_{i+1} \longrightarrow X_i \longrightarrow B_i, 
\] 
where $X \wequi \lim_i X_i$, $B_i$ is simply connected, $X_0 = *$ and the connectivity of $B_i$ tends to infinity with $i$. Moreover each $B_i$ is a grouplike commutative monoid (so in particular nilpotent). Lemma~\ref{lemm:Lret-fiber} shows that these fiber sequences are preserved by $L_\ret$, and Lemma~\ref{lemm:Lret-conn} shows that each $\mathrm{L}_\ret B_i$ is simply connected, and has connectivity tending to infinity with $i$. Since $L_\ret$ preserves finite products, $L_\ret B_i$ is still a grouplike commutative monoid, and so nilpotent. By induction and appeal to \cite[Corollary 3.3.7]{AFHLocalization}, one deduces that each $\mathrm{L}_\ret X_i$ is nilpotent, and so is $X' := \lim_i L_\ret X_i$.

To conclude the proof, we must show that the map $X \to X'$ exhibits the target as $\mathrm{L}_\ret$ of the source.  Thus we must prove that whenever $Y \in \Spc_\ret(k)$ then $\Map(X, Y) \wequi \Map(X', Y)$.  Since $\Spc_\ret(k) \wequi \Shv(k_\ret)$ (by Theorem \ref{thm:main}) which is Postnikov complete (Theorem~\ref{thm:rethypercomplete}), we may assume for this that $Y$ is $n$-truncated for some $n$.  But then for $N$ sufficiently large we get 
\[ 
\Map(X, Y) \wequi \Map(X_N, Y) \wequi \Map(L_\ret X_N, Y) \wequi \Map(X', Y), 
\] 
as required.
\end{proof}

\subsection{Realization of singular varieties} \label{sec:real-realn}
Recall the functor \[ r_\R: \Ft_\R \to (\text{topological spaces}) \] sending a scheme $X$ to its set of real points $X(\R)$, with the strong topology.
It preserves disjoint unions, open and closed immersions, Zariski coverings, finite products (reduce to the affine case using the previous properties) and fiber products (reduce to the separated (e.g. affine) case using the previous properties, then embed the fiber product into the product).
\begin{lem} \label{lemm:rR-abstract-blowup}
Let
\begin{equation*}
\begin{CD}
W @>>> Y \\
@VVV @VVV \\
Z @>>> X
\end{CD}
\end{equation*}
be an abstract blowup square in $\Ft_\R$.
Then the image under $r_\R$ is a pushout square and a homotopy pushout square.
\end{lem}
\begin{proof}
We have bijections of sets \[ r_\R(X) \wequi r_\R(Z) \amalg (r_\R(X) \setminus r_\R(Z)) \wequi r_\R(Z) \amalg (r_\R(Y) \setminus r_\R(W)). \]
Since $r_\R(W) \wequi r_\R(Z) \times_{r_\R(X)} r_\R(Y)$ we deduce that the map \[ X' := r_\R(Z) \amalg_{r_\R(W)} r_\R(Y) \to r_\R(X) \] is bijective.
Since $r_\R(Z) \amalg r_\R(Y) \to X'$ is surjective and $r_\R(Z) \amalg r_\R(Y) \to X$ is closed (use Lemma \ref{lemm:prop-proj} below), $X' \to X$ is also closed.\NB{If $f: A \to B$ is surjective, $g: B \to C$ such that $gf$ is closed then for $Y \subset B$ closed also $g(Y) = g(f(f^{-1}(Y)))$ is closed.}
Thus $X' \to X$ is continuous, bijective and closed, whence an isomorphism; consequently the square is a pushout square.
Since the horizontal maps are cofibrations (this follows from\cite{Lojasiewicz}\NB{more explicit reference?}), the pushout is a homotopy pushout.
\end{proof}

\begin{lem} \label{lemm:prop-proj}
If $f: X \to Y \in \Ft_\R$ is proper, then $r_\R(f)$ is closed.
\end{lem}
\begin{proof}
The problem being local on $Y$, we may assume $Y$ affine.
Let $U \subset X$ be affine.
By Chow's lemma (see e.g. \cite[Corollaire 1.4]{deligne-nagata-compact}), we can find $X'_U \to X$ with $X'_U \to X$ a blowup with center outside of $U$, and $X'_U \to Y$ projective.
Apply this to all open subsets $U$ in an affine open covering of $X$ and take the disjoint union of the resulting schemes.
We obtain $X' \to X$ such that $X' \to Y$ is projective and $r_\R(X') \to r_\R(X)$ is surjective.
To prove that $r_\R(f)$ is closed it thus suffices to prove that $r_\R(X' \to Y)$ is closed, i.e., we may assume that $X$ is projective.
In other words we are given a factorization of $f$ as $X \hookrightarrow \P^n \times Y \to Y$.
Since $r_\R$ preserves closed immersions, we need only treat the map $\P^n \times Y \to Y$.
This case holds because $r_\R$ preserves finite products and $r_\R(\P^n)$ is compact.
\end{proof}

We can left Kan extend the functor \[ \Ft_\R \xrightarrow{r_\R} (\text{topological spaces}) \to \Spc \] to obtain \[ r_\R: {\mathrm P}(\Ft_R) \to \Spc, \] still preserving finite products and coproducts.
\begin{proposition} \label{prop:rR-cdh}
The functor $r_\R: {\mathrm P}(\Ft_\R) \to \Spc$ factors through $\Shv_{cdh}(\Ft_\R)$.
\end{proposition}
\begin{proof}
Via Lemma \ref{lemm:rR-abstract-blowup}, this is immediate from the definition of cdh descent in terms of abstract blowup squares \cite{Voevodskycd}.
\end{proof}

\begin{cor}
Consider the adjunction \[ e: {\mathrm P}(\Sm_R) \leftrightarrows {\mathrm P}(\Ft_\R): e^*. \]
Then $r_\R$ inverts the counit $ee^* \to \id$.
\end{cor}
\begin{proof}
Since finite type schemes are cdh-locally smooth \cite[Lemma 4.3]{VoevodskyNisvscdh}, the map $ee^* \to \id$ is a cdh-local equivalence.
The result thus follows from Proposition \ref{prop:rR-cdh}.
\end{proof}

\begin{rem} \label{rmk:rR-invert}
It is clear that $r_\R: {\mathrm P}(\Ft_\R) \to \Spc$ preserves the generating $\aone$-equivalences $X \times \aone \to X$, and also colimits.
Hence it preserves the strong saturation of the union of cdh equivalences and generating $\aone$-equivalences.
\end{rem}

\begin{ex} \label{ex:Omega-G-1}
Let $G$ be a split reductive group and $X = \Omega_\Gm {\mathrm L}_{mot} G \in \Spc(\R)$.
By \cite[Theorem 15]{bachmann-grassmann}, we know that $X \stackrel{mot}{\wequi} Gr_G$.
Here $Gr_G$ is a singular (ind-)scheme, which we make into a presheaf on $\Sm_\R$ by restriction.
Using Remark \ref{rmk:rR-invert}, we find that $r_\R(X) \wequi r_\R(Gr_G)$, where on the right hand side we just mean the real points of $Gr_G$ with the strong topology.
\end{ex}

\begin{rem} \label{rmk:R-C2}
The analog of Remark \ref{rmk:rR-invert} for complex realization and the $h$-topology also holds.\footnote{It is clear that $r_\C$ preserves surjections of finite type schemes. With this input one may copy the above proof to see that $r_\C$ inverts cdh equivalences, which is all we need for the rest of this remark.}\NB{I used to have a reference for this but cannot find it anymore.}
Now write \[ r_{C_2}: \Ft_\R \to \Spc^{C_2} \wequi \Fun(Orb_{C_2}, \Spc) \] for the $C_2$-equivariant realization functor, sending $X \in \Ft_\R$ to the space $X(\C)$ with the $C_2$-action by Galois conjugation.
Thus as a presheaf, $r_{C_2}(X)(C_2/C_2) = r_\R(X)$ and $r_{C_2}(X)(C_2/e) = r_{\C}(X)$.
It follows that the left Kan extension \[ r_{C_2}: {\mathrm P}(\Ft_\R) \to \Spc^{C_2} \] also inverts cdh-equivalences, $\aone$-equivalences, and their strong saturation.
\end{rem}

\begin{ex}
Continuing with Example \ref{ex:Omega-G-1}, using Remark \ref{rmk:R-C2} we find that \[ r_{C_2}(\Omega_\Gm G) \wequi r_{C_2}(Gr_G) \wequi \Omega_\sigma r_{C_2}(G). \]
Here $\Omega_\sigma$ denotes loops with respect to the sign representation sphere $S^\sigma \in \Spc^{C_2}_*$, and the last equivalence follows from \cite[Theorems 1.1 and 1.2]{CrabbMitchell}.
\end{ex}

\begin{rem}
Using the discussion in Example \ref{ex:Omega-G-1}, Theorem 6.6(3) of \cite{ABHWhitehead} implies \cite[Proposition 3.6]{HahnWilson}.
\end{rem}

\subsubsection*{Real realization and rét-localization}
\begin{theorem} \label{thm:Lret-rR}
The two functors \[ r_\R, L_\ret: \Spc(\R) \to \Spc (\wequi \Shv_\ret(\R) \wequi \Spc_\ret(\R)) \] are canonically equivalent.
\end{theorem}
\begin{proof}
The functor $r_\R: \Spc(\R) \to \Spc$ inverts real étale equivalences (since real étale covers are étale maps that are surjective on real points, this follows from the proofs of \cite[Lemma 5.4, Theorem 5.5]{dugger2001hypercovers-topology}) and hence factors through $L_\ret$.
Since $\Spc_\ret(\R) \wequi \Spc$, we thus have two colimit preserving endofunctors of $\Spc$.
As such they are determined by where they send $*$, and hence are equivalent (both preserving $*$).
\end{proof}

\subsection{Real \'etale realization of Eilenberg--Mac Lane spaces}
Our aim in this section is to study the result of applying $L_\ret$ to the motivic Eilenberg--Mac Lane spaces $K(\Z(m), n)$ and $K(\tilde\Z(m), n)$.
Since they are defined over $\Q$ and stable under base change, applying $L_\ret$ produces a constant sheaf, corresponding to the space obtained from real realization $r_\R$ (see Theorem \ref{thm:Lret-rR}).
Since $r_\R$ preserves finite products, these are grouplike commutative monoids and even $\Z$-modules.
In other words, they are generalized Eilenberg--Mac Lane spaces and so split according to their homotopy groups.
The same is thus true for $L_\ret$ over any field.
Moreover, $L_\ret K(\Z(m), n+1) \wequi B L_\ret K(\Z(m), n)$ and so on.
To describe $L_\ret K(\Z(m), n)$ for all $n \ge m$ it will thus suffice to describe the constant sheaves $\pi_* L_\ret K(\Z(m),m)$.
The ring structure on the motivic Eilenberg--Mac Lane spectrum provides us with multiplication maps $K(\Z(m), m) \wedge K(\Z(n), n) \to K(\Z(m+n), m+n)$ which turns the bigraded abelian group $\pi_* L_\ret(K(\Z(\star),\star))$ into a commutative ring.
Similar observations hold for $K(\tilde Z(m), n)$.

\begin{theorem} \label{thm:EM}
	Work over a field $k$.
	For $n \ge m$, the sheaves $L_\ret K(\Z(m), n)$ and $L_\ret K(\tilde\Z(m), n)$ are constant sheaves on $k_\ret$ (in addition to being connective $\Z$-modules).
	
	We have $\ul\pi_2 L_\ret K(\tilde \Z(2), 2) \wequi \Z$; pick a generator $t$.
	There is additionally a canonical element $\rho \in \ul\pi_0 L_\ret K(\tilde\Z(1), 1)$.
	We have \[ \ul\pi_* L_\ret K(\Z(\star), \star) \wequi \Z[\rho, t]/(2\rho) \quad\text{and}\quad \ul\pi_* L_\ret K(\tilde\Z(\star), \star) \wequi \Z[\rho, t]/(2t\rho). \] \NB{pictures?}
\end{theorem}

\begin{rem}
We know that $K(\Z(m), 2m)$ is the infinite symmetric power of $(\P^1)^{\wedge n}$.
Using Remark \ref{rmk:R-C2}, we see that $r_{C_2}(K(\Z(m), 2m)$ is the infinite symmetric power of $\P^1(\C)^{\wedge n}$, with the canonical action.
By an equivariant Dold--Thom theorem, this is the same as the equivariant ``Eilenberg-Mac Lane space'' $\Omega^\infty \Sigma^{m+m\sigma} H\underline{\Z}$.
Theorem \ref{thm:EM} is thus equivalent to determining the fixed points of this $C_2$-space.
\end{rem}

\begin{rem}[A question of O. Röndigs]
By the above theorem we have $K_1 := L_\ret K(\Z(1), 1) \wequi \Z/2$ (on a generator called $\rho$) and $K_2 := L_\ret K(\Z(2), 2) \wequi \Z/2 \oplus \Z[2]$ (on generators called $\rho^2$ and $t$).
It follows that $K_1 \otimes_\Z K_1 \wequi \Z/2 \oplus \Z/2[1]$.
The multiplicative structure yields a map $K_1 \otimes_\Z K_1 \to K_2$, one component of which takes the form $\delta: \Z/2[1] \to \Z[2]$.
There are two such maps (in $\Z$-modules), the zero map and the integral Bockstein.
Our proof of Theorem \ref{thm:EM} shows that $\delta$ is indeed the Bockstein.
(Equivalently, the generator of the summand $\Z/2[1]$ in $K_1 \otimes_\Z \Z/2$ has non-vanishing square.)
\end{rem}

\subsubsection*{Finite coefficients}
We begin with the following computation.
\begin{proposition} \label{prop:EM-finite}
	Let $p$ be an odd prime.
	We have \[ \ul\pi_* L_\ret K(\Z/2(\star), \star) \wequi \Z/2[\rho,\tau] \quad\text{and}\quad \ul\pi_* L_\ret K(\Z/p(\star), \star) \wequi \Z/p[t]. \]
	Here $\rho \in \pi_0 L_\ret K(\Z/2(1), 1)$, $\tau \in \pi_1 L_\ret K(\Z/2(1), 1)$ and $t \in \pi_2 L_\ret K(\Z/p(2), 2)$.
\end{proposition}
\begin{proof}\NB{references to Bloch-Kato?}
	The key idea for both assertions is the same: we can explicitly determine $L_\ret^\naive K$ (where $K$) is the relevant motivic Eilenberg--Mac Lane space), namely it is constant of the claimed form.
	Since constant real étale sheaves are $\A^1$-invariant (Proposition \ref{prop:A1-inv}), this must also be $L_\ret K$.
	
	We begin with $K(\Z/2(\star), \star)$.
	It is well-known that $\pi_* K(\Z/2(\star), \star) \wequi K_\star^M/2[\tau]$.
	For any real closed field $r$ and any $i \ge 0$ we have $K_i^M(r)/2 \wequi H^i_\et(r, \Z/2) \wequi H^i(B\Z/2, \Z/2) = \Z/2$.
	The same holds for sections over any henselian local ring with real closed residue field, by rigidity  \cite[Theorem 3.4.4]{EHIK}.
	Consequently the associated real étale sheaves of the homotopy sheaves are of the claimed (in particular constant) form.
	
	The argument for odd primes is the same, using that $\pi_* K(\Z/2(\star), \star) \wequi H^{\star - *}_\et(\ph, \mu_p^{\otimes \star}) \wequi \Z/p[t]$ for some element $0 \ne t \in H^2(r, \mu_p^{\otimes 2})$.
	(Noting that $r$ has $p$-étale cohomological dimension $0$ for this one need only observe that $\mu_p^{\otimes 2} \wequi \Z/p$ whereas $\mu_p(\R) = \{1\}$.)
\end{proof}

\subsubsection*{Analyzing the slice filtration using the stabilization theorem}
\begin{lem} \label{lemm:freudenthal}
	Let $k$ be a perfect field of exponential characteristic $e$ and let $E \in \Sigma^{2n,n} \SH(k)^\veff[1/e]$ (for some $n \ge 1$).
	Then the canonical map $\Sigma^\infty \Omega^\infty E \to E$ has fiber in $\Sigma^{4n,2n} \SH(k)^\veff[1/e]$.
\end{lem}
\begin{proof}
	We know that $\Sigma^\infty \Omega^\infty \in \Sigma^{2n,n} \SH(k)^\veff$ \cite[Proposition 3.2.12]{ABHFreudenthal} and hence the fiber lives in $\Sigma^{2n-1,n} \SH(k)^\veff$.
	In order to show that the fiber lies in $\Sigma^{4n,2n} \SH(k)^\veff$, it will thus be enough to show this after applying $\Omega^\infty$.
	(Indeed by \cite[Proposition 3.2.12]{ABHFreudenthal} again, $\Omega^\infty$ commutes with the ``$S^{4n-,2n}$-truncation'', but this truncation also lies in $\Sigma^{2n-1,n} \SH(k)^\veff$ by what we just said. It remains to observe that the restriction of $\Omega^\infty$ to this subcategory is conservative.)
	We can write $\Omega^\infty \Sigma^\infty \Omega^\infty E \wequi \Omega^\infty E \times F$, where $F$ is $\Omega^\infty$ of the fiber (indeed since $E' := \Omega^\infty \Sigma^\infty \Omega^\infty E$ is a monoid we can add the canonical map $F \to E'$ and the map $\Omega^\infty E \to E'$ given as $\Omega^\infty$ of a unit map. This results in a map of grouplike commutative monoids $F \times \Omega^\infty E \to E'$ over $\Omega^\infty E$ which induces an equivalence of the fibers, and hence is an equivalence).
	Applying the Freudenthal suspension theorem (infinite loop version) \cite[Theorem 6.3.4]{ABHFreudenthal} to $\Omega^\infty E$ we see that the fiber of $\Omega^\infty E \to \Omega^\infty \Sigma^\infty \Omega^\infty E$ lies in $O(S^{4n-1,2n})$.
	But this fiber is $\Omega F$, yielding the desired conclusion.
\end{proof}

\begin{lem} \label{lemm:EM-technical}
	We have $0 = \ul\pi_i L_\ret K(\Z[1/2](3), 6)$ for $i \le 5$ and $\ul\pi_3 L_\ret K(\Z(2), 4) = 0$.
\end{lem}
\begin{proof}
	Applying Lemma \ref{lemm:freudenthal} to $E = \Sigma^{6,3} H\Z[1/2]$ we deduce that $\Sigma^\infty K(\Z[1/2](3), 6) \to \Sigma^{6,3} H\Z[1/2]$ has fiber in $\SH(k)_{\ge 6}$.
	It follows that $\ul\pi_i \Sigma^\infty K(\Z[1/2](3), 6)[\rho^{-1}] \wequi \ul\pi_i \Sigma^{6,3}H\Z[1/2,\rho^{-1}] = 0$ (note that multiplication by $h = 1 + \langle -1 \rangle$ is an isomorphism on $H\Z[1/2]$, whence multiplication by $0=\rho \cdot h$ is an isomorphism on the right hand side).
	The left hand term is the same as $\ul\pi_i \Sigma^\infty_\ret L_\ret K(\Z[1/2](3), 6) \wequi \ul\pi_i L_\ret \Sigma^\infty K(\Z[1/2](3), 6)$ by the main result of \cite{bachmann-real-etale}.
	But since $L_\ret K(\Z[1/2](3), 6)$ is a grouplike commutative monoid, the map $\ul\pi_i L_\ret K(\Z[1/2](3), 6) \to \ul\pi_i \Sigma^\infty_\ret L_\ret K(\Z[1/2](3), 6)$ is a split injection, whence the first result.
	
	The second result is proved similarly: $\Sigma^\infty K(\Z(2), 4) \to \Sigma^{4,2} H\Z$ has fiber in $\SH(k)_{\ge 4}$ and so after inverting $\rho$ we get an isomorphism on $\pi_3$.
	It remains to observe that $\pi_* H\Z[1/\rho] \wequi \Z/2[t^2]$ with $|t^2|=2$.
	(The fiber sequence $H\Z \xrightarrow{h} H\Z \to H\Z/2$ yields $L_\ret H\Z/2 \wequi L_\ret H\Z \oplus \Sigma L_\ret H\Z$. We have $\pi_* L_\ret H\Z = 0$ for $i<0$ and $\pi_{**} L_\ret H\Z/2 \wequi \pi_{**}H\Z/2[\rho^{-1}] = \Z/2[\tau,\rho^{\pm}]$ and hence $\pi_* L_\ret H\Z/2 \wequi \Z/2[t]$; the claim follows.)
\end{proof}

\begin{lem} \label{lemm:ret-pi4}
	We have $\ul\pi_4 L_\ret K(\Z[1/2](2), 4) \wequi \Z[1/2]$.
\end{lem}
\begin{proof}
	Recall the slice filtration $f_\bullet \mathrm{KGL} = \Sigma^{2\bullet,\bullet} \mathrm{kgl} \to \mathrm{KGL}$, with cofibers given by $\Sigma^{2i,i} H\Z$.\NB{ref?}
	Inverting $2$ and taking infinity loops we obtain cofiber sequences of grouplike commutative monoids in $\Spc(k)$ as follows
	\begin{enumerate}[noitemsep,topsep=1pt]
		\item $F_1 \to F_0 \to K(\Z[1/2], 0)$
		\item $F_2 \to F_1 \to K(\Z[1/2](1), 2)$
		\item $F_3 \to F_2 \to K(\Z[1/2](2), 4)$
		\item $F_4 \to F_3 \to K(\Z[1/2](3), 6)$.
	\end{enumerate}
	Applying $L_\ret$ preserves these cofiber sequences of grouplike commutative monoids, and so we can use long exact sequences to analyze the homotopy groups.
	To do this, observe the following:
	\begin{enumerate}[noitemsep,topsep=1pt]
		\item $L_\ret K(\Z[1/2], 0) = a_\ret \Z[1/2]$ is discrete.
		\item $L_\ret K(\Z[1/2](1), 2) \wequi (L_\ret B\Gm) \otimes \Z[1/2] = 0$ (recall that $L_\ret B\Gm \wequi \P^\infty(\R) \wequi K(\Z/2, 1)$).
		\item $L_\ret K(\Z[1/2](3), 6)$ is $5$-connected by Lemma \ref{lemm:EM-technical}.
		\item $L_\ret F_4$ is $4$-connected.
	\end{enumerate}
	We shall give a proof of (4) at the end.
	Combining points (3,4) and fiber sequence (4) shows that $F_3$ is $4$-connected.
	Fiber sequence (3) now shows that $\ul\pi_4 L_\ret K(\Z[1/2](2), 4) \wequi \ul\pi_4 L_\ret  F_2$.
	Fiber sequence (2) and point (2) show that $\ul\pi_4 L_\ret  F_2 \wequi \ul\pi_4 L_\ret  F_1$, and similarly fiber sequence (1) and point (1) show that  $\ul\pi_4 L_\ret  F_1 \wequi \ul\pi_4 L_\ret  F_0$.
	But $L_\ret F_0 \wequi B\mathrm{O}[1/2] \times \Z[1/2]$, so that $\ul\pi_4L_\ret F_0 = \Z[1/2]$ by Bott periodicity.\NB{ref?}

	Now we prove (4). The argument is very similar to the ones in Lemma \ref{lemm:EM-technical}. We have $F_4 \in O(S^{8,4})$ and hence, using that $F_4$ is a commutative monoid and Lemma \ref{lemm:freudenthal} we obtain a split injection $\ul\pi_i L_\ret F_4 \to \ul\pi_i L_\ret f_4 \mathrm{KGL}[1/2]$ for $i < 8$. The right hand side vanishes e.g. since $\eta$ acts by zero (since it does on $f_4 \mathrm{KGL}$) but also by an isomorphism (since $L_\ret(\eta) = 2$).
\end{proof}

\begin{proposition}\label{prop:EM-Q}
	We have $\ul\pi_* L_\ret K(\Q(\star), \star) \wequi \Q[t]$ for some $t \in \ul\pi_2 L_\ret K(\Q(2), 2)$.
\end{proposition}
In other words, $L_\ret K(\Q(m), m) = 0$ for $m$ odd, and $L_\ret K(\Q(m),m) \wequi K_\ret(\Q,m)$ for $m$ even.
\begin{proof}
	We may as well prove that $L_\ret K(\Q(m), 2m) = 0$ for $m$ odd, and $L_\ret K(\Q(m),2m) \wequi K_\ret(\Q,2m)$ for $m$ even, and show that the generators are related by multiplication.
	Recall\NB{ref?} the splitting \[ B\GL_\Q \wequi {\prod_{m \ge 0}}' K(\Q(m), 2m), \] where $\prod'$ denotes the weak product (filtered colimit of finite products).
	Applying $L_\ret$ to this equivalence and using that the real points of $\GL$ are given by the orthogonal group $\mathrm{O}$ we obtain the right hand equivalence in the following \[ {\prod_{m \ge 0}}' K(\Q,4m) \wequi B\mathrm{O}_\Q \wequi {\prod_{m \ge 0}}' L_\ret K(\Q(m), 2m). \]
	The left hand equivalence is just real Bott periodicity.
	Recalling that $B\GL_\Q = \Omega^\infty \mathrm{KGL}_\Q$ is a ring, we see that both sides of this equivalence are rings, and the equivalence is compatible with the ring structure.\footnote{Note that the ring structure on $\Sigma^{2*,*} H\Z$ can be \emph{defined} as the ring structure on $s_* \mathrm{KGL}$. This implies that the ring structure on the homotopy groups of the right hand side is compatible with the ring structure we have been using.}
	On homotopy groups, the ring structure on the left hand side is polynomial on a generator $t$ in degree $4$.\NB{ref?}
	The same must thus be true on the right hand side.
	It follows that there is precisely on $m \ge 0$ such that $\ul\pi_4 L_\ret K(\Q(m), 2m) \ne 0$.
	This is where $t$ lives, and our task is to show that $m=2$.
	This is clear from Lemma \ref{lemm:ret-pi4}.
\end{proof}

\subsubsection*{Integral coefficients}
We can now put everything together.
\begin{proof}[Proof of Theorem \ref{thm:EM}.]\NB{All of this seems unreasonably convoluted.}
	Noting that $L_\ret$ preserves cofiber sequences of commutative monoids, we see that $(L_\ret K(\Z(2),2))/p \wequi L_\ret(K(\Z/p(2),2))$, and so on.
	We first want to find $t \in \pi_2 L_\ret K(\Z(2), 2)$ lifting the element from the mod $p$ and rational computations of Propositions \ref{prop:EM-finite} and \ref{prop:EM-Q}.
	Let $E = L_\ret K(\Z(2), 2)(\Q) \wequi L_\ret K(\Z(2), 2)(\R)$.
	We know that $\pi_3(E/p) = 0$ for every $p$, by Proposition \ref{prop:EM-finite}.
	It follows that $\pi_2(E)$ is torsion-free.
	Lemma \ref{lemm:ret-pi4} shows that $\pi_2(E) \otimes \Z[1/2] \wequi \Z[1/2]$.
	Since $\pi_1(E) = 0$ by Lemma \ref{lemm:EM-technical} and $\pi_2(E/2) \wequi \Z/2$ Proposition \ref{prop:EM-finite}, we deduce that $\pi_2(E)/2 = \Z/2$.
	Pick $t_0 \in \pi_2(E)$ with non-zero image in $\pi_2(E)/2 = \Z/2$.
	Note that for an odd prime, $\pi_2(E)/p = \pi_2(E)[1/2]/p$ and hence, an element of $\pi_2(E)$ is divisible by $p$ if and only if its image in $\pi_2(E)[1/2]$ is divisible by $p$.
	It follows that $t_0$ can be divided in $\pi_2(E)$ by odd primes only finitely many times; let $t$ be obtained by doing as many divisions as possible.
	Then $t_0$ and $t$ have the same image in $\Z/2$.
	By what we just said, the image of $t$ in $\pi_2(E)[1/2] \wequi \Z[1/2]$ is not divisible by any odd primes, and hence, generates.
	Thus $t$ has the desired properties.
	
	The element $\rho \in \pi_0 L_\ret K(\tilde \Z(1), 1)$ of course comes from $S^0 \xrightarrow{\rho} \Gm \to K(\tilde \Z(1), 1)$.
	It follows that $\rho \in \pi_0 L_\ret K(\Z(1), 1) \wequi \pi_0 \Gm(\R) \wequi \Z/2$ is a generator, and $2\rho = 0$ in this group.
	We can form maps $\Z[t],\Z[\rho] \to L_\ret K(\Z(\star), \star)$, where the source denotes free $E_1$-rings over $\Z$ (with the homotopy of polynomial rings).
	Using the multiplication in the target we can form a homotopy ring map $\Z[t,\rho] \to L_\ret K(\Z(\star), \star)$ which factors over the cofiber of multiplication by $2\rho$ (since $2\rho =0$ in the target ring).
	We thus get $\Z[t,\rho] \to \Z[t,\rho]/2\rho \xrightarrow{f} L_\ret K(\Z(\star), \star)$, where the long composite is a homotopy ring map but the map $f$ need not be.
	Since the first map is surjective on homotopy groups, $f$ is in fact a ring map on homotopy groups.
	We claim that this is the desired isomorphism, equivalently, that $f$ is an equivalence of spectra.
	To do this it suffices to check on homotopy groups after $\otimes_\Z R$, where $R=\Q$ and $R= \F_p$ for all $p$.
	For $R = \Q$ or $R = \F_p$, $p$ odd, we get $\pi_* (\Z[t,\rho]/2\rho \otimes R) \wequi R[t]$ and thus $\pi_*(f \otimes R)$ is a ring map, whence an isomorphism.
	Now we work modulo $2$.
	We have \[ \Z[t,\rho]/2\rho \wequi \Z[t] \oplus \rho\F_2[t,\rho]. \]
	Using that $\F_2\{a\} \otimes_\Z \F_2 \wequi \F_2\{a\} \oplus \F_2\{\epsilon a\}$ (here $a$ is the name for the generator of some free $\F_2$-module of rank $1$), we can pick the monomial basis in the above expression and find \[ \pi_* (\Z[t,\rho]/2\rho \otimes \F_2) \wequi  \F_2[t, \rho] \oplus \epsilon \cdot \rho\F_2[t, \rho]. \]
	The map $\Z[t,\rho]/2\rho \to \pi_*(\Z[t,\rho]/2\rho \otimes \F_2)$ is onto the first summand of this expression.
	Recall that if $M$ is any $\Z$-module, then $M/2$ acquires the Bockstein endomorphism $\beta: M/2 \to \Sigma M/2$ (it is the composition of the boundary map $M/2 \to \Sigma M$ and the projection $\Sigma M \to \Sigma M/2$); we also denote the effect on homotopy groups by $\beta: \pi_*(M/2) \to \pi_{*-1}(M/2)$.
	For $M=\F_2\{a\}$ the Bockstein is given by $\beta(a) = 0$ and $\beta(\epsilon a) = a$.
	Consequently on $\pi_* (\Z[t,\rho]/2\rho \otimes \F_2)$ we have $\beta(\epsilon x) = x$.
	The map \[ \pi_*(f \otimes \F_2): \pi_* (\Z[t,\rho]/2\rho \otimes \F_2) \to \pi_* L_\ret K(\Z/2(\star), \star) \wequi \F_2[\tau, \rho] \] is a ring map on the subring $\F_2[t, \rho]$ sending $t$ to $\tau^2$ and $\rho$ to $\rho$.
	It is also compatible with Bocksteins, and hence must send $\epsilon \cdot x$ to an element with Bockstein $f(x)$.
	The Bocksteins on $\F_2[\tau,\rho]$ are generated by $\beta(\tau) = \rho$; in other words\footnote{It is well-known that $\beta(\tau) = \rho$. Since $\beta^2 = 0$ (by the Adem relations) it follows that $\beta(\rho) = 0$. The formula now follows from the Leibniz rule (i.e. the Cartan formula).} \[ \beta(\tau^a \rho^b) = \begin{cases} 0 & $a$ \text{ even} \\ \tau^{a-1}\rho^{b+1} & $a$ \text{ odd} \end{cases}. \]
	In particular the kernel $K$ of the Bockstein is contained in the image of $f$.
	Since \[ \beta f(\epsilon \cdot t^a \rho^{b+1}) = f\beta(\epsilon \cdot t^a \rho^{b+1}) = f(t^a\rho^{b+1}) = \tau^{2a}\rho^{b+1} \] we deduce that \[ f(\epsilon \cdot t^a \rho^{b+1}) \in \tau^{2a+1}\rho^b + K. \]
	This implies that $\tau^{2a+1}\rho^b$ is in the image of $f$ (since $K$ is) and, consequently, $f$ is surjective.
	Since this is a map of degreewise finite dimensional $\F_2$-vector spaces of equal dimensions, it must be an isomorphism.
	
	We now deal with $K(\tilde\Z(\star), \star)$.
	We have the pushout square of commutative monoids \cite[Lemma 4.3]{bachmann-etaZ}
	\begin{equation*}
		\begin{CD}
			K(\tilde\Z(\star), \star) @>>> K(\Z(\star), \star) \\
			@VVV                                @VVV \\
			I^{\star} @>>> K_\star^M/2.
		\end{CD}
	\end{equation*}
	This is preserved by $L_\ret$.
	Noting that $L_\ret I^\star \wequi L_\ret^\naive I^\star \wequi \Z$ and $L_\ret K_\star^M/2 \wequi L_\ret^\naive K_\star^M/2 \wequi \Z/2$ and considering the result long exact sequence, the second claim follows.
\end{proof}

\begin{rem}
	The normed structure on $\tilde H\Z$ provides us with a ``twisted cup square'' map \[ K(\tilde\Z(n), 2n) \xrightarrow{u} p_* p^* K(\tilde\Z(n), 2n) \xrightarrow{\mu} K(\tilde\Z(2n), 4n). \]
	Here $p: \Spec(k[i]) \to \Spec(k)$ is the canonical map, $u$ denotes the twisted diagonal, and $\mu$ is the norm.
	Composing with the projection to $K(\tilde\Z(2n), 4n)_{\le 2n} \wequi K(I^n, 2n)$ and applying $\Omega^n L_\ret$ yields a map \[ L_\ret K(\tilde\Z(n), n) \to K(\Z, n). \]
	For $n$ even we wonder if this is the projection to the summand generated by $t^{n/2}$.\NB{justification?}
\end{rem}

\begin{footnotesize}
\bibliographystyle{alpha}
\bibliography{ret}
\end{footnotesize}
\Addresses
\end{document}